\documentclass[a4paper,12pt, oneside]{amsart}

\usepackage{amssymb,amsthm,amsfonts,amsmath,amscd,verbatim}
\usepackage{xy}
\usepackage[hypertex]{hyperref}

\xyoption{all}
\SelectTips{cm}{10}
\date{}
\newtheorem{tr}{Theorem}[section]
\newtheorem*{tr*}{Theorem}

\newtheorem{lemma}[tr]{Lemma}
\newtheorem*{st*}{Statement}
\newtheorem*{pr*}{Proposition}
\newtheorem*{cl}{Claim}
\newtheorem{cor}[tr]{Corollary}
\theoremstyle{definition}
\newtheorem{df}[tr]{Definition}

\newtheorem*{df*}{Definition}
\newtheorem*{not*}{Notation}
\newtheorem*{ex}{Example}

\newtheorem{rem}[tr]{Remark}
\newlength{\myevenmargin}
\def\id{{ id}}
\def\bb#1{\mathbb #1}
\def\cal#1{\mathcal #1}

\def\ra{\rightarrow}
\def\xra{\xrightarrow}

\def\mto{\mapsto}

\def\exp{{\rm exp}}
\def\eqdef{=:}
\def\defeq{:=}

\DeclareMathOperator\Image{Im}
\renewcommand{\Im}{\Image}
\DeclareMathOperator\Ker{Ker}

\DeclareMathOperator\End{End}

\DeclareMathOperator\pr{pr}
 
 \DeclareMathOperator\idd{id}
\DeclareMathOperator\GL{GL}

\DeclareMathOperator\Pic{Pic}

\DeclareMathOperator\Deck{Deck}
\DeclareMathOperator\diag{diag}

\def\Vect{{\sf Vect}}

\renewcommand{\id}{\idd}

\def\ten{\otimes}

\def\XX{\tilde X}
\def\xx{\tilde x}
\let\bigwedge\Lambda

\title{Vector bundles on elliptic curves and factors of automorphy}
\author{Oleksandr Iena}
\address{SISSA\\
via Bonomea, 265\\
34136 Trieste\\
Italy}
\email{iena@sissa.it}

\begin{document}
\begin{abstract}
We translate the Atiyah's results on classification of vector bundles on elliptic curves to the language of factors of automorphy.
\end{abstract}
\maketitle

\tableofcontents
\section{Introduction}
\subsection{Motivation}
The problem of classification of vector bundles over an elliptic curve was considered and  completely solved by Atiyah in~\cite{Atiyah}.

For a group $\Gamma$ acting on a complex manifold $Y$, an $r$-dimensional factor of automorphy  is a holomorphic function $f:\Gamma\times Y\ra \GL_r(\bb C)$ satisfying $f(\lambda\mu, y)=f(\lambda, \mu y)f(\mu, y)$. 
Two factors of automorphy $f$  and $f'$ are equivalent if
there exists a holomorphic function $h:Y\ra \GL_r(\bb C)$ such that
\(
h(\lambda y)f(\lambda, y)=f'(\lambda, y)h(y).
\)

Given a complex manifold $X$  and the universal covering $Y\xra{p} X$, let $\Gamma$ be the fundamental group of $X$ acting naturally on $Y$ by deck transformations. Then there is a one-to-one correspondence between equivalence classes of $r$-dimensional factors of automorphy and isomorphism classes of vector bundles on $X$ with trivial pull-back along $p$. In particular, if $Y$ does not possess any non-trivial vector bundles, one obtains a one-to-one correspondence between  equivalence classes of $r$-dimensional factors of automorphy and isomorphism classes of vector bundles on $X$. In particular this is the case for complex tori.

Since it is known that  one-dimensional complex tori correspond to elliptic curves and since
the classification of holomorphic
vector bundles on a projective variety over $\bb C$ is equivalent to the
classification of algebraic vector bundles (cf.~\cite{Serre}), it is possible to formulate the Atiyah's results in the language of factors of automorphy.

This paper is a shortened version of the diploma thesis~\cite{MyGermanDiplom} and aims to give an accessible reference to the proofs of some results definitely known to the experts but still unpublished or difficult to find. The statement of the main result of this manuscript, Theorem~\ref{classmatr}, coincides with the statement of Proposition 1 from~\cite{Pol}, which was given without any proof.

The author thanks Igor Burban, Bernd Kreu\ss ler,  and G\"unther Trautmann, who motivated him to prepare this manuscript.
\subsection{Structure of the paper}
In Section~\ref{section:basic correspondence} we establish a correspondence between vector bundles and factors of automorphy. 
Section~\ref{section:properties} deals with properties of factors of automorphy, in particular we discuss  a correspondence  between operations on vector bundles and operations on factors of automorphy. From  Section~\ref{section:VB on tori} on we restrict ourselves to the case of vector bundles on complex tori. It is shown in Theorem~\ref{tr:A(1)} that to define a vector bundle of rank $r$ on a complex one-dimensional torus is the same as to fix a holomorphic function $\bb C^*\ra \GL_r(\bb C)$. In Section~\ref{section:classification} we first present in Theorem~\ref{degzerobundle} a classification of indecomposable vector bundles of degree zero, using this we give then in Theorem~\ref{classmatr} a complete classification of indecomposable vector bundles of fixed rank and degree in terms of factors of automorphy.

\subsection{Notations and conventions}
Following Atiyah's paper~\cite{Atiyah} we denote by $\cal E(r,d)=\cal E_X(r,d)$ the
set of isomorphism classes of indecomposable vector bundles over $X$ of rank
$r$ and degree $d$.
For a vector bundle $E$ we usually denote  the corresponding locally free sheaf of its sections by $\cal E$.
By $\Vect$ we denote the category of finite dimensional vector
spaces.  For a divisor $D$ we denote by $[D]$ the corresponding line bundle.

\section{Correspondence between vector bundles and factors of automorphy}
\label{section:basic correspondence}
\def\T{({\bf  T})\ }
Let $X$ be a  complex manifold
and let $p:Y\ra X$ be a covering of $X$. Let $\Gamma\subset
\Deck(Y/X)$  be a subgroup in the group of deck transformations
$\Deck(Y/X)$ such that for any two points $y_1$ and $y_2$ with
$p(y_1)=p(y_2)$ there exists an  element $\gamma\in \Gamma$ such that
$\gamma(y_1)=y_2$. In other words, $\Gamma$ acts transitively in each
fiber. We call this property \T.
\begin{rem}
Note that for any two points $y_1$ and $y_2$ there can be only one
$\gamma \in \Deck(Y/X)$ with $\gamma(y_1)=y_2$ (see~\cite{Forst}, Satz
{\bf 4.8}). Therefore, $\Gamma=\Deck(Y/X)$ and the property \T  simply
means that
$p: Y\ra X$ is a normal (Galois) covering.
\end{rem}

We have an action of $\Gamma$ on $Y$:
\[
\Gamma\times Y\ra Y, \quad  y\mto \gamma(y)\eqdef \gamma y.
\]



\begin{df}
A holomorphic function $f:\Gamma\times Y \ra \GL_r(\bb C)$, $r\in \bb
N$ is called
an $r$-dimensional factor of automorphy if it satisfies the relation
\[
f(\lambda\mu,y)=f(\lambda,\mu y)f(\mu,y).
\]
Denote by $Z^1(\Gamma,r)$ the set of all $r$-dimensional factors of
automorphy.
\end{df}
We introduce  the relation $\sim$ on $Z^1(\Gamma, r)$. We say that $f$ is
equivalent to $f'$ if
there exists a holomorphic function $h:Y\ra \GL_r(\bb C)$ such that
\[
h(\lambda y)f(\lambda, y)=f'(\lambda, y)h(y).
\]
We write in this case $f\sim f'$.
\begin{cl}
The relation $\sim$ is  an equivalence relation on $Z^1(\Gamma, r)$.
\end{cl}
\begin{proof}
Straightforward verifications.
\end{proof}

We denote the set of equivalence classes of $Z^1(\Gamma, r)$ with
respect to $\sim$ by $H^1(\Gamma, r)$.

Consider $f \in Z^1(\Gamma, r)$ and a trivial vector bundle
$Y\times\bb C^r\ra Y$. Define a holomorphic action of $\Gamma$
on $Y\times\bb C^r$:
\[
\Gamma\times Y\times\bb C^r\ra Y\times\bb C^r,\quad
(\lambda, y, v)\mto (\lambda y, f(\lambda, y)v)\eqdef\lambda(y, v).
\]
Denote $E(f)=Y\times\bb C^r/\Gamma$ and note that for two equivalent
points $(y, v)\sim_\Gamma (y', v')$ with respect to the action of
$\Gamma$ on $Y\times \bb C^r$ it follows that $p(y)=p(y')$. In fact,
$(y, v)\sim_\Gamma (y', v')$ implies in particular that $y=\gamma y'$
for some $\gamma\in \Gamma$ and by the definition of deck
transformations $p(y)=p(\gamma y')=p(y')$. Hence the
projection $Y\times\bb C^r\ra Y$ induces the map
\[
\pi:E(f)\ra X,\quad 
[y, v]\mto p(y).
\]
We equip $E(f)$ with the quotient topology.
\begin{tr}
$E(f)$ inherits a complex structure from $Y\times \bb C^r$ and the
  map $\pi:E(f)\ra X$ is a holomorphic vector bundle on $X$.\label{AutoB}
\end{tr}

\begin{proof}
First we prove that $\pi$ is a topological vector bundle. Clearly
$\pi$ is a continuous map.
Consider the commutative diagram
\[
\xymatrix
{
Y\times\bb C^r\ar[d]\ar[r]&E(f)\ar[d]^\pi\\
Y\ar[r]^p&X.
}
\]
Let $x$ be a point of $X$. Since $p$ is a covering, one can
choose an open
neighbourhood $U$ of $x$ such that its preimage is a
disjoint union of open sets biholomorphic to U, i. e.,
$p^{-1}(U)=\bigsqcup\limits_{i\in \cal I} V_i$, $p_i\defeq p|_{V_i}:V_i\ra U$ is a biholomorphism for each $i\in
\cal I$. For each pair $(i,j)\in \cal I\times\cal I$ there exists a unique
$\lambda_{ij}\in \Gamma$ such that
$\lambda_{ij}p_j^{-1}(x)=p_i^{-1}(x)$ for all $x\in U$. This follows
from the property \T.

We have  $\pi^{-1}(U)=((\bigsqcup\limits_{i \in \cal I}V_i)\times
\bb C^r)/\Gamma$.

Choose  some $i_U\in \cal I$. Consider the holomorphic map
\[
\varphi_U':(\bigsqcup\limits_{i\in \cal I}
V_i)\times \bb C^r\ra U\times \bb C^r,\quad
(y_i, v)\mto (p(y_i), f(\lambda_{i_Ui},  y_i)v),\  y_i\in V_i.
\]

Suppose that $( y_i, v')\sim_\Gamma( y_j, v)$. This means
\[
( y_i, v')=\lambda_{ij}( y_j, v)=(\lambda_{ij} y_j, f(\lambda_{ij}, y_j)v).
\]
 Therefore,
\begin{align*}
\varphi_U'( y_i, v')=&(p( y_i), f(\lambda_{i_Ui},
 y_i)v')=(p(\lambda_{ij} y_j), f(\lambda_{i_Ui},
\lambda_{ij} y_j)f(\lambda_{ij}, y_j)v)=\\&(p( y_j),
f(\lambda_{i_Uj}, y_j)v)=\varphi_U'( y_j, v).
\end{align*}
Thus
$\varphi_U'$ factorizes through $((\bigsqcup\limits_{i\in \cal  I}
V_i)\times \bb C^r)/\Gamma$, i. e., the map
\[
\varphi_U:((\bigsqcup\limits_{i\in \cal I}
V_i)\times \bb C^r)/\Gamma\ra U\times \bb C^r,\quad
[( y_i, v)]\mto (p( y_i), f(\lambda_{i_Ui},  y_i)v),\  y_i\in V_i
\]
is well-defined and continuous. We claim that $\varphi_U$ is bijective.

Suppose $\varphi_U([( y_i, v')])=\varphi_U([( y_j, v)])$, where
$y_i\in V_i$, $y_j\in V_j$. By definition this is
equivalent to $(p( y_i), f(\lambda_{i_Ui},  y_i)v')=(p( y_j),
f(\lambda_{i_Uj},  y_j)v)$, which means $ y_i=\lambda_{ij} y_j$ and
\begin{align*}
f(\lambda_{i_Ui},
\lambda_{ij} y_j)v'=&f(\lambda_{i_Ui}, y_i)v'=f(\lambda_{i_Uj}, y_j)v=\\
&f(\lambda_{i_Ui}\lambda_{ij},
 y_j)v=f(\lambda_{i_Ui},\lambda_{ij} y_j)f(\lambda_{ij}, y_j)v.
\end{align*}
We conclude  $v'=f(\lambda_{ij}, y_j)v$ and $[( y_i, v')]=[( y_j, v)]$,
which means injectivity of $\varphi_U$.

At the same time for each
element $(y,v)\in U\times\bb C^r$ one has
\begin{align*}
\varphi&_U([(p_{i}^{-1}(y),
f(\lambda_{i_Ui},  p_{i}^{-1}(y))^{-1}v)])=\\&(pp_{i}^{-1}(y),
f(\lambda_{i_Ui}, p_i^{-1}(y))f(\lambda_{i_Ui},
p_{i}^{-1}(y))^{-1}v)=(y, v),
\end{align*} i. e., $\varphi_U$ is surjective and we obtain that
$\varphi_U$ is a  bijective map.

This means, that $\varphi_U$ is a trivialization for $U$ and that
$\pi:E(f)\ra X$ is a (continuous) vector bundle.
If $U$ and $V$ are two neighbourhoods of $X$ defined as above for
which $E(f)|_{U}$, $E(f)|_{V}$ are trivial, then the corresponding
transition function is
\[
\varphi_{U}\varphi_{V}^{-1}:(U\cap V)\times\bb C^r\ra (U\cap
V)\times \bb C^r, \quad
(x,v)\mto(x, g_{UV}(x)v),
\]
where $g_{UV}:U\cap V\ra\GL_r(\bb C)$ is a cocycle defining $E(f)$.
But from the construction of $\varphi_U$  it follows
that
\[
g_{UV}(x)=f(\lambda_{i_{U}i_{V}},
p_{i_{V}}^{-1}(x) ).
\]
Therefore, $g_{UV}$ is a
holomorphic map,  hence  $\varphi_{U}\varphi_{V}^{-1}$ is also a
holomorphic map. Thus the maps $\varphi_U$ give $E(f)$ a complex
structure. Since $\pi$ is locally a projection, one sees that
$\pi$ is a holomorphic map.
\end{proof}
\begin{rem}
Note that $p^*E(f)$ is isomorphic to $ Y\times\bb C^r$. An isomorphism
can be given by the map
\[
p^*E(f)\ra  Y\times \bb C^r,\quad
( y, [\tilde{y}, v])\mto ( y,f(\lambda, \tilde y)v),\quad \lambda\tilde y= y.
\]
\end{rem}
Now we have the map from $Z^1(\Gamma, r)$ to the set
$K_r=\{[E]\ |\ p^*(E)\simeq  Y\times\bb C^r\}$  of isomorphism classes of
vector bundles of rank $r$ over $X$
with trivial pull back with respect to $p$.
\[
\phi': Z^1(\Gamma, r)\ra K_r;\quad f\mto [E(f)].
\]
\begin{tr}\label{automain}
Let $K_r$ denote the set of  isomorphism classes  of vector bundles
of rank $r$
on $X$ with trivial pull back with respect to $p$. Then the map
\[
H^1(\Gamma, r)\ra K_r,\quad [f]\mto [E(f)]
\]
is a bijection.
\end{tr}
\begin{proof} This proof generalizes the proof
    from~\cite{Lange} given only for line bundles.

Consider the map $\phi':Z^1(\Gamma, r)\ra K_r$ and let $f$ and $f'$ be two
equivalent $r$-dimensional factors of automorphy. It means that there exists a holomorphic
function $h: Y\ra \GL_r(\bb C)$ such that
\[
f'(\lambda, y)=h(\lambda y)f(\lambda,
 y)h( y)^{-1} .
\]
Therefore, for two neighbourhoods $U$, $V$ constructed as
above we have the following relation for cocycles corresponding to $f$ and $f'$.
\begin{align*}
g'_{UV}(x)=&f'(\lambda_{UV},p^{-1}_{i_{V}}(x))= h(\lambda_{UV}p^{-1}_{i_{V}}(x)) f(\lambda_{{U}{V}},
p_{i_{V}}^{-1}(x))
h( p^{-1}_{i_{V}}(x) )^{-1} =\\ &h(p^{-1}_{i_{U}}(x))g_{UV}(x)
h( p^{-1}_{i_{V}}(x) )^{-1} =  h_{U}(x) g_{UV}(x)
h_{V}(x)^{-1},
\end{align*}
where $\lambda_{UV}=\lambda_{i_U i_V}$,   $h_{U}(x)= h( p^{-1}_{i_{U}}(x)
)$  and  $h_{V}(x)= h( p^{-1}_{i_{V}}(x)
)$. We obtained
\[
g'_{UV}= h_{U} g_{UV}
h_{U}^{-1},
\]
which is exactly the condition for two cocycles to
define isomorphic vector bundles. Therefore,  $E(f)\simeq E(f')$ and
it means that $\phi'$ factorizes through $H^1(\Gamma, r)$, i. e., the
map
\[
\phi: H^1(\Gamma, r)\ra K_r;\quad [f]\mto [E(f)]
\]
is well-defined.

It remains to construct the inverse map.
Suppose $E\in K_r$, in other words $p^*(E)$ is the trivial  bundle of
rank $r$
over  $Y$. Let $\alpha:p^*E\ra  Y\times\bb C^r$ be a
trivialization. The action of $\Gamma$ on $ Y$ induces a holomorphic
action of $\Gamma$ on $p^*E$ :
\[
\lambda( y,e)\defeq(\lambda y,e) \text{ for $( y,e)\in p^*E=Y\times_X E$.}
\]
 Via $\alpha$ we get for every
$\lambda\in \Gamma$ an automorphism $\psi_\lambda$ of the trivial
bundle $ Y\times \bb C^r$. Clearly  $\psi_\lambda$ should be of the
form
\[
\psi_\lambda
( y,v)=(\lambda y, f(\lambda, y)v),
\]
where $f:\Gamma\times Y\ra\GL_r(\bb C)$ is a holomorphic map. The equation for the action
$\psi_{\lambda\mu}=\psi_\lambda\psi_\mu$ implies that f should be an
$r$-dimensional factor of automorphy.

Suppose $\alpha'$ is  an another trivialization of $p^*E$. Then there
exists a holomorphic map $h: Y\ra \GL_r(\bb C)$ such that
$\alpha'\alpha^{-1}( y,v)=( y, h( y)v)$. Let $f'$ be a factor of
automorphy corresponding to $\alpha'$. From
\begin{align*}
(\lambda y,f'(\lambda, y)v)=&\psi_\lambda'( y,v)=\alpha'\lambda\alpha'^{-1}( y,v)=\alpha'\alpha^{-1}\alpha\lambda\alpha^{-1}\alpha\alpha'^{-1}( y,v)=\\&\alpha'\alpha^{-1}\psi_\lambda(\alpha'\alpha^{-1})^{-1}( y,v)=\alpha'\alpha^{-1}\psi_\lambda( y,h( y)^{-1}v)=\\&\alpha'\alpha^{-1}(\lambda y,f(\lambda, y)h( y)^{-1}v)=(\lambda y,h(\lambda y)f(\lambda, y)h( y)^{-1}),
\end{align*} we obtain $f'(\lambda, y)=h(\lambda y)f(\lambda, y)h( y)^{-1}$. The
last means that $[f]=[f']$, in other words,
the class of a factor of automorphy
in $H^1(\Gamma, r)$ does not depend on the trivialization and we get a
map $K_r\ra H^1(\Gamma, r)$. This map is the inverse of $\phi$.
\end{proof}

Let $X$ be a connected complex manifold, let $p:\XX\ra X$ be a
universal covering of $X$, and let
$\Gamma=\Deck(Y/X)$.  Since universal
coverings are normal coverings, $\Gamma$ satisfies the property \T (see~\cite[Satz~5.6]{Forst}). Moreover, $\Gamma$ is isomorphic to the
fundamental group $\pi_1(X)$ of $X$ (see~\cite[Satz~5.6]{Forst}).
An isomorphism is given as follows.

Fix $x_0\in X$ and $\xx_0\in \XX$ with $p(\xx_0)=x_0$. We define a map
\[
\Phi:\Deck(\XX/ X)\ra \pi_1(X, x_0)
\]
as follows.  Let $\sigma\in \Deck(\XX/X)$ and $v:[0;1]\ra \XX$ be a
curve with $v(0)=\xx_0$ and $v(1)=\sigma(\xx_0)$. Then a curve
\[
pv:[0;1]\ra X, \quad t\mto pv(t)
\] is such that $pv(0)=pv(1)=x_0$. Define $\Phi(\sigma)\defeq [pv]$,
where $[pv]$ denotes a homotopy class of $pv$. The map $\Phi$ is well
defined and is an isomorphism of groups.

So we can identify $\Gamma$ with $\pi_1(X)$. Therefore, we have an
action of $\pi_1(X)$ on $\XX$ by deck transformations.

Consider an element $[w]\in \pi_1(X,x_0)$ represented by a path
$w:[0;1]\ra X$.
We denote $\sigma=\Phi^{-1}([w])$.
Consider any $\xx_0\in X$ such that $p(\xx_0)=w(0)=w(1)$, then the path  $w$ can be
uniquely lifted to the path
\[
v:[0;1]\ra \XX
\]
with $v(0)=\xx_0$ (see~\cite{Forst}, Satz {\bf 4.14}). Denote
$\xx_1=v(1)$. Then $\sigma$ is a unique element in $\Deck(\XX/X)$ such
that $\sigma(\xx_0)=\xx_1$.
This gives a description of the action of $\pi_1(X,x_0)$ on $\XX$.

Now we have a corollary to Theorem~\ref{automain}.
\begin{cor}
Let $X$ be a connected complex manifold, let $p:\XX \ra X$ be the
universal covering, let $\Gamma$ be the fundamental group of $X$
naturally acting on $\XX$ by deck transformations. As above, $H^1(\Gamma, r)$
denotes the set of equivalence classes of $r$-dimensional factors of
automorphy
\[
\Gamma\times\XX\ra \GL_r(\bb C).
\]
Then there is a bijection
\[
H^1(\Gamma, r)\ra K_r,\quad [f]\mto E(f),
\]
where $K_r$ denotes the set  of isomorphism classes of vector bundles
of rank $r$
on $X$ with trivial pull back with respect to $p$.
\end{cor}

\section{Properties of factors of automorphy}
\label{section:properties}
\begin{df}\label{thetafunction}
Let $f:\Gamma\times Y\ra \GL_r(\bb C)$ be an $r$-dimensional factor of automorphy. A holomorphic
function $s:Y\ra\bb C^r$ is called an $f$-theta function
if it satisfies
\[
s(\gamma y)=f(\gamma, y)s(y) \text{ for all $\gamma\in \Gamma$, $y\in
  Y$.}
\]
\end{df}

\begin{tr}\label{sections}
Let $f:\Gamma\times Y\ra \GL_r(\bb C)$ be an $r$-dimensional factor of
automorphy. Then there is a one-to-one correspondence between sections of $E(f)$ and
$f$-theta functions.
\end{tr}
\begin{proof}
Let $\{V_i\}_{i\in \cal I}$ be a covering of $Y$ such that $p$
restricted to $V_i$ is a homeomorphism. Denote $\varphi_i\defeq
(p|_{V_i})^{-1}$, $U_i\defeq p(V_i)$. Then $\{U_i\}$ is a covering of
$X$ such that $E(f)$ is trivial over each $U_i$.

Consider  a section of $E(f)$ given by  functions $s_i:U_i\ra \bb C^r$
satisfying
\[
s_i(x)=g_{ij}(x)s_j(x)\text{ for $x\in U_i\cap U_j$,}
\]where
\[
g_{ij}(x)=f(\lambda_{U_i U_j}, \varphi_j(x)),\quad x\in U_i\cap U_j
\]
is  a cocycle defining $E(f)$ (see the proof of Theorem~\ref{automain}).
Define $s:Y\ra \bb C^r$ by
$s(\varphi_i(x))\defeq s_i(x)$. To prove that this  is
well-defined we need to show that $s_i(x)=s_j(x)$ when
$\varphi_i(x)=\varphi_j(x)$. But since $\varphi_i(x)=\varphi_j(x)$ we
obtain $\lambda_{U_i U_j}=1$. Therefore,
\begin{align*}
s_i(x)=g_{ij}(x)s_i(x)=f(\lambda_{U_i U_j},\varphi_j(x))s_j(x)=f(1, \varphi_j(x))s_j(x)=s_j(x).
\end{align*}
For any $\gamma\in \Gamma$ for any point $y\in Y$ take
$i, j\in \cal I$ and $x\in X$ such that $y=\varphi_j(x)$ and
$\gamma y=\gamma\varphi_j(x)=\varphi_i(x)$. Thus $\gamma =
\lambda_{U_iU_j}$ and one obtains
\begin{align*}
s(\gamma y)=&s(\varphi_i(x))=s_i(x)=g_{ij}(x)s_j(x)=\\&f(\lambda_{U_iU_j},
\varphi_j(x))s_j(x)=f(\gamma, y)s(\varphi_j(x))= f(\gamma, z)s(y).
\end{align*}
In other words, $s$ is an $f$-theta function.

Vice versa, let $s:Y\ra \bb C^r$ be an $f$-theta function. We
define $s_i:U_i\ra \bb C^r$ by $s_i(x)\defeq s(\varphi_i(x))$. Then
for a point $x\in U_i\cap U_j$ we have
\begin{align*}
s_i(x)=&s(\varphi_i(x))=s(\lambda_{U_iU_j}\varphi_j(x))=\\&f(\lambda_{U_iU_j},\varphi_j(x))s(\varphi_j(x))=g_{ij}(x)s_j(x),
\end{align*}
which means that the functions $s_i$ define a section of $E(f)$. The described
correspondences are clearly inverse to each other.
\end{proof}

The following statement will be useful in the sequel.
\begin{tr}\label{extension}
Let
\(
f(\lambda, y)=
\begin{pmatrix}
f'(\lambda, y)&\tilde{f}(\lambda, y)\\
0&f''(\lambda,  y)
\end{pmatrix}
\) be  an $r'+r''$-dimensional factor of automorphy, where
$f'(\lambda,y)\in \GL_{r'}(\bb C)$, $f''(\lambda, y)\in \GL_{r''}(\bb C)$. Then

(a)$f':\Gamma\times  Y\ra \GL_{r'}(\bb C)$
and  $f'':\Gamma\times  Y\ra \GL_{r''}(\bb C)$ are $r'$ and
$r''$-dimensional factors of automorphy respectively;

(b) there is an extension of vector
bundles
\[
\xymatrix{
0\ar[r]&E(f')\ar[r]^i&E(f)\ar[r]^\pi&E(f'')\ar[r]&0
}
.\]
\end{tr}
\begin{proof}
The statement of (a) follows from straightforward verification.
To prove (b) we define  maps $i$ and $\pi$ as follows.
\begin{align*}
i:&E(f')\ra E(f),\quad [ y,v]\mto [ y,\begin{pmatrix}v\\0\end{pmatrix}],\quad v\in \bb
C^{r'},\quad \begin{pmatrix}v\\0\end{pmatrix}\in \bb C^{r'+r''}\\
\pi:&E(f)\ra E(f''),\quad [ y,\begin{pmatrix}v\\w\end{pmatrix}]\ra
   [ y,w],\quad v\in \bb C^{r'},\quad w\in \bb C^{r''}
\end{align*}
Since $[\lambda y,f'(\lambda, y)v]$ is mapped via $i$ to
$[\lambda y,\begin{pmatrix}f'(\lambda, y)v\\0\end{pmatrix}]=[\lambda y,f(\lambda, y)\begin{pmatrix}v\\0\end{pmatrix}]$,
one concludes that $i$ is well-defined. Analogously, since $[\lambda y,f''(\lambda, y)w]=[ y,w]$ one sees
that $\pi$ is well-defined. Using the charts from the proof of (\ref{AutoB}) one easily sees that
the defined maps are holomorphic.

Notice that $i$ and $\pi$ respect fibers, $i$ is injective  and $\pi$ is surjective in each
fiber. This proves the statement.
\end{proof}

Now we recall one standard construction from linear algebra.
Let A be an $m\times n$ matrix. It represents some morphism $\bb
C^n\ra \bb C^m$ for fixed standard bases in $\bb C^n$ and $\bb C^m$.

Let $\cal F:\Vect^p\ra \Vect$ be a covariant functor. Let $A_1,\dots,
A_p$ be the matrices representing morphisms $\bb
C^n_1\stackrel{f_1}{\ra} \bb C^m_1, \dots, \bb
C^n_p\stackrel{f_p}{\ra} \bb C^m_p$ in standard bases.

If for each  object $\cal F(\bb C^m)$ we fix some basis, then the
matrix corresponding to the morphism $\cal F(f_1,\dots, f_p)$ is
denoted by $\cal F(A_1,\dots, A_p)$. Clearly it satisfies
\[
\cal F(A_1B_1,\dots, A_pB_p)=\cal F(A_1,\dots, A_p)\cal F(B_1,\dots, B_p).
\]

In this way  $A\ten B$, $S^q(A)$, $\bigwedge^q(A)$ can be defined. As
$\cal F$ one considers the functors
\[
\_\ten\_:\Vect^2\ra \Vect,\quad
S^n:\Vect \ra \Vect,\quad
\bigwedge:\Vect\ra \Vect
\]
respectively.

Recall that every holomorphic functor $\cal F:\Vect^n\ra \Vect$  can be canonically extended to the
category of vector bundles of finite rank over $X$. By abuse of notation we will denote the extended functor by $\cal F$ as well.

\begin{tr}\label{funct}
Let $\cal F:\Vect^n\ra \Vect$ be a covariant holomorphic functor.
Let $f_1,\dots,f_n$ be $r_i$-dimensional factors of automorphy.
Then $f=\cal F(f_1,\dots,f_n)$ is  a factor of automorphy
defining $\cal F(E(f_1), \dots, E(f_n))$.
\end{tr}
\begin{proof}One clearly has
\begin{align*}
\cal F(f_1,\dots, f_n)(\lambda\mu, y)=&\cal
F(f_1(\lambda\mu, y),\dots, f_n(\lambda\mu, y))=\\
&\cal F(f_1(\lambda,\mu  y)f_1(\mu,  y),\dots,f_n(\lambda,\mu
 y)f_n(\mu,  y))=\\
&F(f_1(\lambda,\mu  y),\dots, f_n(\lambda,\mu
 y))F(f_1(\mu,  y),\dots, f_n(\mu,  y))=\\
&\cal F(f_1,\dots, f_n)(\lambda,\mu
 y)\cal F(f_1,\dots, f_n)(\mu,  y).
\end{align*}

Since $(f_1,\dots, f_n)$ represents an isomorphism in $\Vect^n$, $\cal
F(f_1,\dots, f_n)$ also represents an isomorphism $\bb C^r\ra \bb C^r$
for some $r\in \bb N$. Therefore, $f$ is an $r$-dimensional factor of automorphy.

Since $f=\cal F(f_1,\dots, f_n)$, the same holds for cocycles  defining the
corresponding vector bundles, i. e., $g_{U_1U_2}=\cal
F({g_1}_{U_1U_2},\dots, {g_n}_{U_1U_2})$, where ${g_i}_{U_1U_2}$ is a
cocycle defining $E(f_i)$. This is exactly the condition
$E(f)=\cal F (E(f_1),\dots, E(f_n))$.
\end{proof}
For example for $\cal F=\_\ten\_:\Vect^2\ra \Vect$ we get the following obvious corollary.
\begin{cor}
Let $f':\Gamma\times  Y\ra \GL_{r'}(\bb C)$
and  $f'':\Gamma\times  Y\ra \GL_{r''}(\bb C)$ be two factors of automorphy. Then
$f=f'\ten f'':\Gamma\times  Y\ra \GL_{r'r''}(\bb C)$ is also a factor
of automorphy. Moreover, $E(f)\simeq E(f')\ten E(f'')$.\label{tensor}
\end{cor}

It is not essential that the functor in Theorem~\ref{funct} is
covariant. The following theorem is a generalization of
Theorem~\ref{funct}.
\begin{tr}
Let $\cal F:\Vect^n\ra \Vect$ be a holomorphic functor. Let $\cal F$ be covariant in $k$ first variables
and  contravariant in $n-k$ last variables.
Let $f_1,\dots,
f_n$ be $r_i$-dimensional factors of automorphy. Then $f=\cal F(f_1,\dots,
f_k,f_{k+1}^{-1},\dots, f_n^{-1})$ is  a factor of automorphy defining $\cal F(E(f_1), \dots, E(f_n))$.
\end{tr}
\begin{proof}
The proof is analogous to the proof of Theorem~\ref{funct}.
\end{proof}

\section{Vector bundles on complex tori}
\label{section:VB on tori}
\subsection{One dimensional complex tori}\label{chaptertori}
Let $X$ be a complex torus, i. e., $X=\bb C/\Gamma$, $\Gamma=\bb Z\tau+\bb Z$, $\Im\tau > 0$. Then the universal covering is $\XX=\bb
C$, namely
\[
\pr:\bb C\ra \bb C/\Gamma,\quad x\mto [x].
\]
We have an action of $\Gamma$ on $\bb C$:
\[
\Gamma\times \bb C\ra \bb C, \quad (\gamma, y)\mto \gamma + y.
\]
Clearly $\Gamma$ acts on $\bb C$ by deck transformations and
satisfies the property \T.

 Since $\bb C$ is a non-compact Riemann surface, by \cite[Theorem~30.4, p. 204]{Forst}, there are only trivial bundles on $\bb C$. Therefore,  we have a one-to-one correspondence  between
classes of isomorphism of vector bundles of rank $r$  on $X$  and
equivalence classes of
factors of automorphy
\[
f:\Gamma\times\bb C\ra\GL_r(\bb C).
\]

As usually, $V_a$ denotes the standard parallelogram constructed at
point $a$, $U_a$ is  the image of $V_a$ under the
projection, $\varphi_a:U_a\ra V_a$ is the local inverse of the projection.

\begin{rem}
Let $f$ be an $r$-dimensional  factor of automorphy. Then
\[
g_{ab}(x)=f(\varphi_a(x)-\varphi_b(x), \varphi_b(x))
\] is a cocycle
defining $E(f)$. This follows from the construction of the cocycle in
the proof of Theorem~\ref{automain}.
\end{rem}

\begin{ex}
There are  factors of automorphy corresponding to classical theta
functions.
For any theta-characteristic $\xi=a\tau+b$, where $a,b\in \bb R$, there is a holomorphic function
$\theta_\xi:\bb C\ra \bb C$ defined by
\begin{align*}
\theta_\xi(z)=\sum\limits_{n\in \bb Z}\exp(\pi i (n+a)^2\tau)\exp(2\pi
i (n+a)(z+b)),
\end{align*}
which satisfies
\[
\theta_\xi(\gamma+ z)=\exp(2\pi i a \gamma -\pi i
p^2\tau - 2 \pi i p(z+\xi))\theta_\xi(z)=e_\xi(\gamma,
  z)\theta_\xi(z),
\]
where $\gamma=p\tau+q$ and  $e_\xi(\gamma,z)=\exp(2\pi i a\gamma-\pi i
p^2\tau - 2 \pi i p(z+\xi))$. Since
\[
e_\xi(\gamma_1+\gamma_2,z)=e_\xi(\gamma_1,\gamma_2+z)e_\xi(\gamma_2,z),
\]
we conclude that $e_\xi(\gamma, z)$ is a factor of automorphy.

By Theorem~\ref{sections} $\theta_\xi(z)$ defines a section of $E(e_\xi(\gamma, z))$.


For more information on classical theta functions see~\cite{MumThetaI, MumThetaII, MumThetaIII}.
\end{ex}

\begin{tr}\label{degreetheta}
$\deg E(e_\xi)=1$.
\end{tr}
\begin{proof}
We know that sections of $E(e_\xi)$ correspond to $e_\xi$ - theta functions. The classical $e_\xi$-theta
function $\theta_\xi(z)$ defines
a section $s_\xi$ of $E(e_\xi)$. Since $\theta_\xi$ has only
simple zeros and  the set of zeros of
$\theta_\xi(z)$ is $\frac{1}{2}+\frac{\tau}{2}+\xi+\Gamma$, we conclude
that $s_\xi$ has exactly one zero at point
$p=[\frac{1}{2}+\frac{\tau}{2}+\xi]\in X$.  Hence by~\cite[p.~136]{GrHar}
we get $E(e_\xi)\simeq [p]$  and thus  $\deg
E(e_\xi)=1$.
\end{proof}
\begin{tr}\label{tr:shift}
Let $\xi$ and $\eta$ be two theta-characteristics. Then
\[
E(e_\xi)\simeq t^*_{[\eta-\xi]}E(e_\eta),
\]
where
\(
t_{[\eta-\xi]}:X\ra X,\quad x\mto x+ [\eta-\xi]
\)
is the
translation by $[\eta-\xi]$.
\end{tr}
\begin{proof}
As in the proof of Theorem~\ref{degreetheta} $E(e_\xi)\simeq [p]$ and
$E(e_\eta)=[q]$ for $p=[\frac{1}{2}+\frac{\tau}{2}+\xi]$ and
$q=[\frac{1}{2}+\frac{\tau}{2}+\eta]$. Since $t_{[\eta-\xi]}p=q$, we
get
\[
E(e_\xi)\simeq [p]\simeq t^*_{[\eta-\xi]}[q]\simeq t^*_{[\eta-\xi]}E(e_\eta),
\]
which completes the proof.
\end{proof}

Now we are going to investigate the extensions of the type
\[
0\ra X\times\bb C\ra E\ra X\times\bb C\ra 0.
\]
In this case the transition functions are given by matrices of the
type
\[
\begin{pmatrix}
1&*\\
0&1
\end{pmatrix}.
\]
 and $E$ is isomorphic to $E(f)$ for some factor of automorphy $f$ of the form
\(
f(\lambda,\xx)=
\begin{pmatrix}
1&\mu(\lambda,\xx)\\0&1
\end{pmatrix}
\). Note that the condition for $f$ to be a factor of automorphy  in this case is equivalent
to the condition
\[
\mu(\lambda+\lambda',\xx)=\mu(\lambda,\lambda'+\xx)+\mu(\lambda',\xx),
\]
where we use the additive notation for the group operation since $\Gamma$
is commutative.

\begin{tr}
$f$ defines trivial bundle if and only if
  $\mu(\lambda,\xx)=\xi(\lambda\xx)-\xi(\xx)$ for some holomorphic
  function $\xi:\bb C\ra\bb C$.
\label{triv}
\end{tr}
\begin{proof}
We know that $E$ is trivial if and only if
$h(\lambda\xx)=f(\lambda,\xx)h(\xx)$ for some holomorphic
function $h:\XX\ra \GL_2(\bb C)$. Let
\(
h=
\begin{pmatrix}
a(\xx)&b(\xx)\\
c(\xx)&d(\xx)
\end{pmatrix}
\), then the last condition is

\begin{align*}
\begin{pmatrix}
a(\lambda\xx)&b(\lambda\xx)\\
c(\lambda\xx)&d(\lambda\xx)
\end{pmatrix}
=
&\begin{pmatrix}
1&\mu(\lambda,\xx)\\
0&1
\end{pmatrix}
\begin{pmatrix}
a(\xx)&b(\xx)\\
c(\xx)&d(\xx)
\end{pmatrix}
=\\
&\begin{pmatrix}
a(\xx)+c(\xx)\mu(\lambda,\xx)&b(\xx)+d(\xx)\mu(\lambda,\xx)\\
c(\xx)&d(\xx)
\end{pmatrix}.
\end{align*}
In particular it means $c(\lambda\xx)=c(\xx)$ and
$d(\lambda\xx)=d(\xx)$, i. e., $c$ and $d$ are doubly periodic functions
on $\XX=\bb C$, so they should be constant, i. e., $c(\lambda,\xx)=c\in \bb
C$, $d(\lambda,\xx)=d\in \bb C$.

Now we have
\begin{align*}
a(\xx)+c\mu(\lambda,\xx)&=a(\lambda\xx)\\
b(\xx)+d\mu(\lambda,\xx)&=b(\lambda\xx)
\end{align*} which implies
\begin{align*}
c\mu(\lambda,\xx)&=a(\lambda\xx)-a(\xx)\\
d\mu(\lambda,\xx)&=b(\lambda\xx)-b(\xx).
\end{align*}
Since $\det h(\xx)\neq 0$ for all $\xx\in\XX=\bb
C$ one of the numbers $c$ and $d$ is not equal to zero. Therefore, one
concludes that $\mu(\lambda,\xx)=\xi(\lambda\xx)-\xi(\xx)$ for some
holomorphic function $\xi:\XX=\bb C\ra \bb C$.

Now suppose $\mu(\lambda,\xx)=\xi(\lambda\xx)-\xi(\xx)$ for some
holomorphic function $\xi:\bb C\ra \bb C$. Clearly for
\(
h(\xx)=
\begin{pmatrix}
1&\xi(\xx)\\
0&1
\end{pmatrix}
\) one has that $\det h(\xx)=1\neq 0$ and
\begin{align*}
f(\lambda,\xx)h(\xx)=&
\begin{pmatrix}
1&\mu(\lambda,\xx)\\
0&1
\end{pmatrix}
\begin{pmatrix}
1&\xi(\xx)\\
0&1
\end{pmatrix}
=\\
&
\begin{pmatrix}
1&\xi(\xx)+\mu(\lambda,\xx)\\
0&1
\end{pmatrix}
=
\begin{pmatrix}
1&\xi(\lambda\xx)\\
0&1
\end{pmatrix}
=
h(\lambda\xx).
\end{align*}
We have shown, that $f$ defines the trivial bundle.
This proves the statement of the theorem.
\end{proof}
\begin{tr}
Two factors of automorphy
\(
f(\lambda,\xx)=
\begin{pmatrix}
1&\mu(\lambda,\xx)\\
0&1
\end{pmatrix}
\) and
\(
f'(\lambda,\xx)=
\begin{pmatrix}
1&\nu(\lambda,\xx)\\
0&1
\end{pmatrix}
\) defining non-trivial bundles are equivalent if and only if
\begin{align*}
\mu(\lambda,\xx)-k\nu(\lambda,\xx)=\xi(\lambda\xx)-\xi(\xx),\quad k\in
\bb C, \quad k\neq 0
\end{align*} for some holomorphic function $\xi:\bb C=\XX\ra \bb C$.
\end{tr}
\begin{proof}
Suppose the factors of automorphy
\(
f(\lambda,\xx)=
\begin{pmatrix}
1&\mu(\lambda,\xx)\\
0&1
\end{pmatrix}
\) and
\(
f'(\lambda,\xx)=
\begin{pmatrix}
1&\nu(\lambda,\xx)\\
0&1
\end{pmatrix}
\) are
equivalent. This  means that
$f(\lambda,\xx)h(\xx)=h(\lambda\xx)f(\lambda,\xx)$ for some
holomorphic function $h:\bb C=\XX\ra \GL_2(\bb C)$. Let
\[
h(\xx)=
\begin{pmatrix}
a(\xx)&b(\xx)\\
c(\xx)&d(\xx)
\end{pmatrix}.
\] The condition for equivalence of $f$ and $f'$ can be rewritten in
the following way:
\begin{align*}
\begin{pmatrix}
1&\mu(\lambda,\xx)\\
0&1
\end{pmatrix}
\begin{pmatrix}
a(\xx)&b(\xx)\\
c(\xx)&d(\xx)
\end{pmatrix}
&=
\begin{pmatrix}
a(\lambda\xx)&b(\lambda\xx)\\
c(\lambda\xx)&d(\lambda\xx)
\end{pmatrix}
\begin{pmatrix}
1&\nu(\lambda,\xx)\\
0&1
\end{pmatrix}
\\
\begin{pmatrix}
a(\xx)+c(\xx)\mu(\lambda,\xx)&b(\xx)+d(\xx)\mu(\lambda,\xx)\\
c(\xx)&d(\xx)
\end{pmatrix}
&=
\begin{pmatrix}
a(\lambda\xx)&a(\lambda\xx)\nu(\lambda,\xx)+b(\lambda\xx)\\
c(\lambda\xx)&c(\lambda\xx)\nu(\lambda,\xx) + d(\lambda\xx)
\end{pmatrix}.
\end{align*}
This  leads to the system of equations
\begin{align*}
\begin{cases}
a(\xx)+c(\xx)\mu(\lambda,\xx)=a(\lambda\xx)\\
b(\xx)+d(\xx)\mu(\lambda,\xx)=a(\lambda\xx)\nu(\lambda,\xx)+b(\lambda\xx)\\
c(\xx)=c(\lambda\xx)\\
d(\xx)=c(\lambda\xx)\nu(\lambda,\xx) + d(\lambda\xx).
\end{cases}
\end{align*}
The third equation means that $c$ is a double periodic
function. Therefore, $c$ should be a constant function.

If $c\neq 0$
from the first and the last equations using Theorem~\ref{triv} one
concludes that $f$ and $f'$ define the trivial bundle.

In  the case $c=0$ one has
\begin{align*}
\begin{cases}
a(\xx)=a(\lambda\xx)\\
b(\xx)+d(\xx)\mu(\lambda,\xx)=a(\lambda\xx)\nu(\lambda,\xx)+b(\lambda\xx)\\
d(\xx)=d(\lambda\xx),
\end{cases}
\end{align*}
i. e., as above, $a$ and $d$ are constant and both not equal to zero
since $\det(h)\neq 0$. Finally one concludes that
\begin{align}
d\mu(\lambda,\xx)-a\nu(\lambda,\xx)=b(\lambda\xx)-b(\xx),\quad a,d\in
\bb C, \quad ad\neq 0\label{*}
\end{align}

Vice versa, if $\mu$ and $\nu$ satisfy~(\ref{*})
for
\(
h(\xx)=
\begin{pmatrix}
a&b(\xx)\\
0&d
\end{pmatrix}
\)
we have

\begin{align*}
f(\lambda,\xx)h(\xx)=
&\begin{pmatrix}
1&\mu(\lambda,\xx)\\
0&1
\end{pmatrix}
\begin{pmatrix}
a&b(\xx)\\
0&d
\end{pmatrix}=
\begin{pmatrix}
a&b(\xx)+d\mu(\lambda, \xx)\\
0&d
\end{pmatrix}=\\
&\begin{pmatrix}
a&b(\lambda\xx)+a\nu(\lambda, \xx)\\
0&d
\end{pmatrix}=
\begin{pmatrix}
a&b(\lambda\xx)\\
0&d
\end{pmatrix}
\begin{pmatrix}
1&\nu(\lambda,\xx)\\
0&1
\end{pmatrix}=
h(\lambda\xx)f(\lambda,\xx).
\end{align*}

 This means
that $f$ and $f'$ are equivalent.
\end{proof}
\subsection{Higher dimensional complex tori}
One can also consider higher dimensional complex tori. Let $\Gamma\subset
\bb C^g$ be a lattice,
\[
\Gamma=\Gamma_1\times\dots
\times\Gamma_g,\quad \Gamma_i=\bb Z+ \bb Z\tau_i,\quad \Im \tau >0.
\]
Then as for one dimensional complex tori  we obtain that $X=\bb
C^g/\Gamma$ is a complex manifold. Clearly the map
\[
\bb C^g\ra \bb C^g/\Gamma=X, \quad x\mto [x]
\]
is the universal covering of $X$. Since all vector  bundles on $C^g$
are trivial, we obtain a one-to-one correspondence between equivalence
classes of
$r$-dimensional factors of
automorphy
\[
f:\Gamma\times \bb C^g\ra \GL_r(\bb C)
\]
and vector bundles of rank $r$ on $X$.

Let $\Gamma=\bb Z^g+\Omega \bb Z^g$, where $\Omega$ is a  symmetric
complex
$g\times g$ matrix with positive definite real part. Note that $\Omega$
is a generalization of $\tau$ from  one dimensional case.

For any theta-characteristic $\xi=\Omega a+b$, where $a\in \bb R^g$, $b\in \bb R^g$ there is a holomorphic function
$\theta_\xi:\bb C^g\ra \bb C$ defined by
\begin{align*}
\theta_\xi(z)=\theta^a_b(z, \Omega) =\sum\limits_{n\in \bb
  Z^g}\exp(\ \pi i (n+a)^t\Omega (n+a)\tau\ )\exp(\ 2\pi
i (n+a)^t\Omega(z+b)\ ),
\end{align*}
which satisfies
\[
\theta_\xi(\gamma+ z)=\exp(2\pi i a^t \gamma -\pi i
p^t\Omega p - 2 \pi i p^t(z+\xi))\theta_\xi(z)=e_\xi(\gamma,
  z)\theta_\xi(z),
\]
where $\gamma=\Omega p+q$ and  $e_\xi(\gamma,z)=\exp(2\pi i a^t \gamma -\pi i
p^t\Omega p - 2 \pi i p^t(z+\xi))$.
 Since
\[
e_\xi(\gamma_1+\gamma_2,z)=e_\xi(\gamma_1,\gamma_2+z)e_\xi(\gamma_2,z),
\]
we conclude that $e_\xi(\gamma, z)$ is a factor of automorphy.

As above $\theta_\xi(z)$ defines a section of $E(e_\xi(\gamma, z))$.

For more detailed information on  higher dimensional theta functions
see~\cite{MumThetaI, MumThetaII, MumThetaIII}.

\subsection{Factors of automorphy depending only on the
  $\tau$-direction of the \hbox{lattice $\Gamma$}}
\label{subsection:only TAU factors of automorphy}
Here $X$ is  a complex torus, $X=\bb C/\Gamma$, $\Gamma=\bb
Z\tau+\bb Z$, $\Im\tau > 0$.
Denote $q=e^{2\pi i \tau}$. Consider the canonical projection
\[
\pr:\bb C^*\ra \bb C^*/<q>, \quad u\ra [u]=u<q>.
\]
 Clearly one
can equip $\bb C^*/<q>$ with the quotient
topology. Therefore, there is a natural complex structure on $\bb C^*/<q>$.

Consider the homomorphism
\[
\bb C\stackrel{\exp}{\ra}\bb C^*\stackrel{\pr}{\ra} \bb C^*/<q> ,\quad z\mto e^{2\pi i z}\mto [e^{2\pi i z}].
\]It is clearly surjective. An element $z\in \bb C$ is in the kernel of this homomorphism if and
only if $e^{2\pi i z}=q^k=e^{2\pi i k \tau}$ for some integer $k$. But
this holds if and only if  $z-k\tau\in \bb Z$ or, in other words, if
$z\in \Gamma$. Therefore, the kernel of the map is exactly $\Gamma$,
and we obtain an isomorphism of groups
\[
iso:\bb C/\Gamma\ra \bb C^*/<q>= \bb C^*/\bb Z,\quad [z]\ra [e^{2\pi i z}].
\]

Since the diagram
\[
\xymatrix{
\bb C\ar[r]^{\pr}\ar[d]_{\exp}&\ar@{<->}[d]^{iso} \bb C/\Gamma\\
\bb C^*\ar[r]^{\pr}& \bb C^*/\bb Z\ar@{=}[r]&\bb C^*/<q>
}
\] is commutative, we conclude that the complex structure on $\bb
C^*/<q>$ inherited  from $\bb C/\Gamma$ by the isomorphism $iso$
coincides with  the natural complex structure on $\bb C^*/<q>$. Therefore,
$iso$ is an isomorphism of complex manifolds.
Thus complex tori can be represented as $\bb C^*/<q>$, where $q=e^{2\pi
i \tau}$, $\tau\in \bb C$, $\Im \tau>0$.

So for any complex torus $X=\bb C^*/<q>$ we have a natural surjective holomorphic
map
\[
\bb C^*\ra \bb C^*/<q>=X,\quad u\ra [u].
\]
This map is moreover a covering of $X$. Consider the group $\bb Z$. It
acts holomorphically on $X=\bb C^*$:
\[
\bb Z\times \bb C^*\ra \bb C^*,\quad (n,u)\mto q^nu.
\]
Moreover, since $\pr(q^nu)=\pr(u)$, $\bb Z$ is naturally identified
with a subgroup in the group  of deck transformations $\Deck(X/\bb C^*)$. It is easy
to see that $\bb Z$ satisfies the property \T. We obtain that there is
a one-to-one correspondence between classes of isomorphism of vector
bundles over $X$ and classes of equivalence of
factors of automorphy
\[
f:\bb Z\times\bb C\ra\GL_r(\bb C).
\]

Consider the following action of $\Gamma$ on $\bb C^*$:
\[
\Gamma\times\bb C^*\ra \bb C^*;\quad (\lambda,u)\mto \lambda u\eqdef
e^{2\pi i \lambda}u
\]
Let $A:\Gamma\times \bb C^*\ra \GL_r(\bb C)$ be a holomorphic function
satisfying
\[
A(\lambda+\lambda',u)=A(\lambda,\lambda'u)A(\lambda',u)\eqno{(*)}
\]
for all $\lambda,\lambda'\in \Gamma$. We call such functions $\bb C^*$-factors of automorphy. Consider the map
\[
\id_{\Gamma}\times \exp:\Gamma\times \bb C \ra  \Gamma\times \bb C^*
,\quad
(\lambda,x)\ra(\lambda,e^{2\pi i x})
\]

Then the function
\[
f_A=A\circ(\id_{\Gamma}\times \exp):\Gamma\times \bb C\ra \GL_r(\bb C)
\]
is an $r$-dimensional factor of
automorphy, because
\begin{align*}
f_A(\lambda+\lambda',x)=A(\lambda+\lambda',e^{2\pi i
  x})=A(\lambda,e^{2\pi i \lambda'}e^{2\pi i x})A(\lambda', e^{2\pi i
  x})=\\
A(\lambda,e^{2\pi i (\lambda'+x)})A(\lambda', e^{2\pi i
  x})=f_A(\lambda,\lambda'+x)f_A(\lambda',x).
\end{align*}
So, factors of automorphy on $\bb C^*$ define factors of automorphy on
$\bb C$.

We restrict ourselves to factors of automorphy  $f:\Gamma\times \bb C\ra \GL_r(\bb
C)$ with the property
\begin{align}
\label{onlyM}
f(m\tau+n,x)=f(m\tau,x),\quad m,n\in \bb Z.
\end{align}
It follows from this property that $f(n,x)=f(0,x)=\id_{\bb C^r}$. Therefore,
\[
f(\lambda+k,x)=f(\lambda,k+x)f(k,x)=f(\lambda,k+x)\text{ for all $\lambda
\in \Gamma$, $k\in \bb Z$}
\] and it is possible to define the function
\begin{align*}
A_f:\Gamma\times \bb C^* \ra \GL_r(\bb C),\quad
(\lambda,e^{2\pi i
  x})\mto f(\lambda, x),
\end{align*}
which is well-defined because from $e^{2\pi i x_1}=e^{2\pi i x_2}$
follows $x_1=x_2+k$ for some $k\in \bb Z$ and
$f(\lambda,x_1)=f(\lambda,x_2+k)=f(\lambda,x_2)$.

Consider $A$ with the  property $A(m\tau+n,u)=A(m\tau,
u)\eqdef A(m,u)$. Then clearly $f_A(m\tau+n,u)=f_A(m\tau,
u)$.
So for any $\bb C^*$-factor of automorphy $A:\Gamma\times\bb C^*\ra
\GL_r(\bb C)$ with the property $A(m\tau+n,u)=A(m\tau,u)$ one obtains
the factor of automorphy $f_A$ satisfying~(\ref{onlyM}). We proved the
following
\begin{tr}\label{polauto}
Factors of automorphy $f:\Gamma\times \bb C\ra \GL_r(\bb C)$ with the property~(\ref{onlyM}) are in a one-to-one correspondence with
$\bb C^*$-factors of automorphy with property $A(m\tau+n,u)=A(m\tau,u)$.
\end{tr}

Now we want to translate the conditions for factors of automorphy with the
property~(\ref{onlyM}) to be equivalent in the language of $\bb
C^*$-factors of automorphy with the same property.
\begin{tr}
Let $f$, $f'$ be $r$-factors of automorphy with the property (\ref{onlyM}). Then $f\sim
f'$ if and only if there exists a holomorphic function $B:\bb C^*\ra
\GL_r(\bb C)$ such that
\[
A_f(m,u)B(u)=B(q^{m}u)A_{f'}(m,u))
\]
for $q\defeq e^{2\pi i \tau}$, where $A(m,u)\defeq A(m\tau,u)$. In
this case we also say $A_f$ is equivalent to $A_{f'}$ and write $A_f\sim A_{f'}$.
\end{tr}
\begin{proof}
Let $f\sim f'$. By definition it means that there exists a holomorphic
function $h:\bb C\ra \GL_r(\bb C)$ such that
$f(\lambda,x)h(x)=h(\lambda x)f'(\lambda,x)$. Therefore, from
$f(n,x)h(x)=h(n+x)f'(n,x)$ and $f(n,x)=f'(n,x)=\id_{\bb C^r}$ it
follows $h(x)=h(n+x)$ for all $n\in \bb Z$. Therefore, the function
\[
B:\bb C^*\ra \GL_r(\bb C),\quad
e^{2\pi i x}\mto h(x)
\]
 is well-defined. We have
\begin{align*}
A_f(m,e^{2\pi i x})B(e^{2\pi i
  x})=&f(m\tau,x)h(x)=h(m\tau+x)f'(m\tau,x)=\\&B(e^{2\pi i
  (m\tau+x)})f'(m,e^{2\pi i x})=B(q^me^{2\pi i
  x})A_{f'}(m,e^{2\pi i x}).
\end{align*}
Vice versa, let $B$ be such that
$A_f(m,u)B(u)=B(q^{m}A_{f'}(m,u))$. Define $h=B\circ \exp$. We obtain
\begin{align*}
f(m\tau+n,x)&h(x)=A_f(m\tau+n,e^{2\pi i x})B(e^{2\pi i x})=\\
&B(q^me^{2\pi i x})A_{f'}(m\tau+n,e^{2\pi i x})=B(e^{2\pi
i (m\tau+x)})A_{f'}(m\tau+n,e^{2\pi i x})=\\
&B(e^{2\pi i (m\tau+n+x)})A_{f'}(m\tau+n,e^{2\pi i x})=h(m\tau +n+x)f'(m\tau+n,x),
\end{align*}
which means that $f\sim f'$ and completes the proof.
\end{proof}
\begin{rem}
 The last two  theorems allow us  to embed the set $Z^1(\bb Z, r)$ of factors of automorphy $\bb Z\times X\ra
 \GL_r(\bb C)$ to the set $Z^1(\Gamma, r)$. The embedding is
\[
\Psi:Z^1(\bb Z, r)\ra Z^1(\Gamma, r), \quad f\mto g,\quad g(n\tau+m,x)\defeq f(n,x).
\]
Two factors of automorphy
 from $Z^1(\bb Z, r)$ are equivalent if and only if their images under
 $\Psi$ are equivalent in $Z^1(\Gamma, r)$.
That is why it is enough to consider only factors of automorphy
\[
\Gamma\times \bb C \ra \GL_r(\bb C)
\]
satisfying~(\ref{onlyM}).
\end{rem}
\begin{cor}A factor of automorphy $f$ with property (\ref{onlyM}) is trivial if and
  only if $A_f(m,u)=B(q^mu)B(u)^{-1}$ for some holomorphic function
  $B:\bb C^*\ra \GL_r(\bb C)$.
\end{cor}

\begin{tr}\label{tr:A(1)}
Let $A$ be a $\bb C^*$-factor of automorphy.
$A(m,u)$ is uniquely determined by $A(u)\defeq A(1,u)$.
\begin{align}
A(m,u)&=A(q^{m-1}u)\dots A(qu)A(u),\quad m>0\\
A(-m,u)&=A(q^{-m}u)^{-1}\dots A(q^{-1}u)^{-1},\quad m>0.\label{matdef}
\end{align}
$A(m,u)$ is equivalent to $A'(m,u)$ if and only if
\begin{align}
\label{equivA}
A(u)B(u)=B(qu)A'(u)
\end{align}
 for some holomorphic function  $B:\bb C^*\ra
\GL_r(\bb C)$. In particular
$A(m,u)$ is trivial iff $A(u)=B(qu)B(u)^{-1}$.
\end{tr}
\begin{proof}
Since $A(1,u)=A(u)$ the first formula holds for $m=1$. Therefore,
\[
A(m+1,u)=A(1,q^mu)A(m,u)=A(q^m)A(m,u)
\] and we prove the first formula by induction.

Now $\id=A(0,u)=A(m-m,u)=A(m,q^{-m}u)A(-m, u)$ and hence
\begin{align*}
A(-m,u)=&A(m,q^{-m}u)^{-1}=(A(q^{m-1}q^{-m}u)\dots
A(qq^{-m}u)A(q^{-m}u))^{-1}=\\&A(-m,u)=A(q^{-m}u)^{-1}\dots A(q^{-1}u)^{-1}
\end{align*}
which proves the second formula.

If $A(m,u)\sim A'(m,u)$ then clearly (\ref{equivA}) holds.

Vice versa, suppose $A(u)B(u)=B(qu)A'(u)$. Then
\begin{align*}
&A(m,u)B(u)=A(q^{m-1}u)\dots A(qu)A(u)B(u)=\\&A(q^{m-1}u)\dots A(qu)B(qu)A'(u)=\dots=B(q^mu)A'(q^{m-1}u)\dots A'(qu)A'(u)=\\&B(q^mu)A'(m,u)
\end{align*} for $m>0$.

Since $A(-m,u)=A(m,q^{-m}u)^{-1}$ we have
\begin{align*}
A(-m,u)B(u)=&A(m,q^{-m}u)^{-1}B(u)=(B(u)^{-1}A(m,q^{-m}u))^{-1}=\\
&(B(u)^{-1}A(m,q^{-m}u)B(q^{-m}u)B(q^{-m}u)^{-1})^{-1}=\\
&(B(u)^{-1}B(q^mq^{-m}u)A'(m,q^{-m}u)B(q^{-m}u)^{-1})^{-1}=\\
&(B(u)^{-1}B(u)A'(m,q^{-m}u)B(q^{-m}u)^{-1})^{-1}=\\
&B(q^{-m}u)A'(m,q^{-m}u)^{-1}=B(q^{-m}u)A'(-m,u),
\end{align*}
which completes the proof.
\end{proof}
\begin{rem}
Theorem~\ref{tr:A(1)} means that all the information about a vector bundle of rank $r$ on a complex torus can be encoded by a holomorphic function $\bb C^*\ra \GL_r(\bb C)$.
\end{rem}

For a holomorphic function $A:\bb C^*\ra \GL_r(\bb C)$, let us denote by $E(A)$ the corresponding vector bundle on $X$.

\begin{tr}\label{tensorPol}
Let $A: \bb C^*\ra \GL_n(\bb C)$, $B:\bb C^*\ra \GL_m(\bb C)$ be two
holomorphic maps. Then $E(A)\ten E(B)\simeq E(A\ten B)$.
\end{tr}
\begin{proof}
By theorem \ref{tensor} we have
\[
E(A)\ten E(B)\simeq E(A(n,u))\ten
E(B(n,u))\simeq E(A(n,u)\ten
B(n,u)).
\]
 Since $A(1,u)\ten B(1,u)=A(u)\ten B(u)$, we obtain $E(A)\ten E(B)\simeq E(A\ten
B)$.
\end{proof}

\section{Classification of vector bundles over a complex torus}
\label{section:classification}
Here we work with factors of automorphy depending only on $\tau$, i. e., with holomorphic functions $\bb C^*\ra \GL_r(\bb C)$.

\subsection{Vector bundles of degree zero}
We return  to extensions of the type $0\ra I_1\ra E\ra I_1\ra 0$, where
$I_1$ denotes  the trivial vector bundle of rank $1$.

Theorem~\ref{triv} can be rewritten as follows.

\begin{tr}
\label{trivA}
A function $A(u)=
\begin{pmatrix}
1&a(u)\\
0&1
\end{pmatrix}$
defines  the trivial bundle if and only if $a(u)=b(qu)-b(u)$ for some
holomorphic function $b:\bb C^*\ra \bb C$.
\end{tr}
\begin{cor}
\(
A(u)=
\begin{pmatrix}
1&1\\
0&1
\end{pmatrix}
\) defines a non-trivial vector bundle.
\end{cor}
\begin{proof}
Suppose $A$ defines the trivial bundle. Then $1=b(qu)-b(u)$ for some
holomorphic function $b:\bb C^*\ra \bb C$. Considering the Laurent
series expansion $\sum\limits_{-\infty}^{+\infty}b_ku^k$ of $b$ we obtain
$1=b_0-b_0=0$ which shows that our assumption was false.
\end{proof}

Let $a:\bb C^*\ra \bb C$ be a holomorphic function such that
\(
A_2(u)=
\begin{pmatrix}
1&a(u)\\
0&1
\end{pmatrix}
\) defines non-trivial bundle, i. e., by Theorem~\ref{trivA}, there is no holomorphic
function $b:\bb C^*\ra\bb C$ such that
\[
a(u)=b(qu)-b(u).
\]
Let $F_2$  be the bundle defined by $A_2$. Then by Theorem~\ref{extension} there exists an
exact sequence
\[
0\ra I_1\ra F_2\ra I_1\ra 0.
\]

For $n\geqslant 3$ we define $A_n:\bb C^*\ra \GL_n(\bb C)$,
\[
A_n=
\begin{pmatrix}
1&a&&\\
&\ddots&\ddots&\\
&&1&a\\
&&&1
\end{pmatrix},
\] where empty entries stay for zeros.

Let  $F_n$ be the bundle defined by $A_n$. By (\ref{extension}) one sees that $A_n$ defines the extension
\[
0\ra I_1\ra F_n\ra F_{n-1}\ra 0.
\]
\begin{tr}
$F_n$ is not the trivial bundle.
The extension
\[
0\ra I_1\ra F_n\ra F_{n-1}\ra 0.
\]
is  non-trivial for all $n\geqslant 2$.
\end{tr}
\begin{proof}
Suppose $F_n$ is trivial. Then $A_n(u)B(u)=B(qu)$ for some
$B=(b_{ij})_{ij^n}$. In particular it means $b_{ni}(u)=b_{ni}(qu)$ for
$i=\overline{1,n}$. Let
$b_{ni}=\sum\limits_{-\infty}^{+\infty}b_{k}^{(ni)}u^k$ be the
expansion of $ b_{ni}$ in Laurent series. Then $b_{ni}(u)=b_{ni}(qu)$ implies
$b_k^{(ni)}=q^kb_k^{(ni)}$ for all $k$.

Note that $|q|<1$ because $\tau=\xi+i\eta$, $\eta>0$ and
\[
|q|=|e^{2\pi i
  \tau}|=|e^{2\pi i (\xi+i\eta)}|=|e^{2\pi i\xi}e^{-2\pi
  \eta}|=e^{-2\pi \eta}<1.
\]
Therefore, $b_k^{(ni)}=0$ for $k\neq 0$ and we conclude that $b_{ni}$
should be constant functions.

We also have
\[
b_{n-1i}(u)+b_{ni}a(n)=b_{n-1i}(qu).
\]
Since at least one of $b_{ni}$ is not equal to zero because of
invertibility of $B$, we obtain
\[
a(u)=\frac{1}{b_{ni}}(b_{n-1i}(qu)-b_{n-1i}(u))
\]
for some $i$, which contradicts the choice of $a$. Therefore, $F_n$ is
not trivial.

Assume now, that for some $n > 2$ the extension
\[
0\ra I_1\ra F_{n}\ra F_{n-1}\ra 0
\]
is trivial(for $n=2$ it is not trivial since $F_2$ is not a trivial
vector bundle). This means
$
A_{n}\sim
\begin{pmatrix}
1&0\\
0&A_{n-1}
\end{pmatrix}
$, i. e., there exists a holomorphic function $B:\bb C^*\ra \GL_n(\bb C)$,
$B=(b_{ij})_{i,j}^n$ such that
\[
A_{n}(u)B(u)=B(qu)
\begin{pmatrix}
1&0\\
0&A_{n-1}
\end{pmatrix}.
\]
Considering the elements of the first and second columns we obtain for
the first column
\begin{align*}
&b_{n1}(u)=b_{n1}(qu),\\
&b_{i1}(u)+b_{i+11}(u)a(u)=b_{i1}(qu), \quad i< n
\end{align*}
and for the second column
\begin{align*}
&b_{n2}(u)=b_{n2}(qu),\\
&b_{i2}(u)+b_{i+12}(u)a(u)=b_{i2}(qu),\quad i<n.
\end{align*}
For the first column as above considering Laurent series  we have that $b_{n1}$ should be a constant
function. If $b_{n1}\neq 0$ it follows
\[
a(u)=\frac{1}{b_{n1}}(b_{n-11}(qu)-b_{n-11}(u)),
\]
which contradicts the choice of $a$. Therefore, $b_{n1}=0$ and
$b_{n-11}(qu)=b_{n-11}(u)$, in other words $b_{n-11}$ is a
constant function. Proceeding by induction one obtains that $b_{11}$ is a
constant function and $b_{i1}=0$ for $i>1$.

For the second column absolutely analogously we obtain a similar
result: $b_{12}$ is constant, $b_{i2}=0$ for $i>1$. This contradicts the invertibility of $B(u)$ and  proves the statement.
\end{proof}

\begin{cor}\label{Atiyahbundles}
The vector bundle $F_n$ is the only indecomposable vector bundle of rank $n$ and
degree $0$ that has non-trivial sections.
\end{cor}
\begin{proof}
This follows from \cite[Theorem~5]{Atiyah}.
\end{proof}
So we have  that the vector bundles $F_n=E(A_n)$ are exactly $F_n$'s
defined by Atiyah in~\cite{Atiyah}.

\begin{rem}
Note that constant matrices $A$ and $B$ having the same Jordan normal
form are equivalent. This is clear because $A=SBS^{-1}$ for some
constant invertible matrix $S$, which means that $A$ and $B$ are equivalent.
\end{rem}

Consider an upper triangular matrix $B=(b_{ij})_1^n$ of the following
type:
\begin{align}
b_{ii}=1,\quad b_{ii+1}\neq 0.\label{oneequiv}
\end{align}
It is easy to see that this
matrix is equivalent to the upper triangular matrix $A$,
\begin{align}
a_{ii}=a_{ii+1}=1,\quad a_{ij}=0,\quad j\neq i+1, \quad j\neq i.\label{one}
\end{align}
In fact, these
matrices have the same characteristic polynomial $(t-1)^n$ and the
dimension of the eigenspace corresponding to the eigenvalue $1$ is equal to $1$ for both
matrices. Therefore, $A$ and $B$ have the same Jordan form. By Remark
above we obtain that $A$ and $B$ are equivalent.
We proved the following:
\begin{lemma}\label{samegord}
A matrix satisfying (\ref{oneequiv}) is equivalent to the matrix defined
by (\ref{one}). Moreover, two matrices of the type (\ref{oneequiv})
are equivalent, i. e., they define two isomorphic vector bundles.
\end{lemma}

\begin{tr}\label{Th9}
$F_n\simeq S^{n-1}(F_2)$.
\end{tr}
\begin{proof}
We know that $F_2$ is defined by the constant matrix
$A_2=
\begin{pmatrix}
1&1\\
0&1
\end{pmatrix}
$. We know by Theorem~\ref{funct} that $S^n(F_2)$ is
defined by $S^n(A_2)$. We calculate $S^n(f_2)$ for $n\in \bb
N_0$. Since $f_2$ is a constant matrix, $S^n(f_2)$ is also a constant
matrix defining a map $S^n(\bb C^2)\ra S^n(\bb C^2)$. Let $e_1$, $e_2$
be the standard basis of $\bb C^2$, then $S^n(\bb C)$ has a basis
\[
\{e_1^ke_2^{n-k}|\  k=n, n-1, \dots, 0\}.
\]
Since $A_2(e_1)=e_1$ and $A_2(e_2)=e_1+e_2$, we conclude that
$e_1^ke_2^{n-k}$ is mapped to
\begin{align*}
A_2(e_1)^kA_2(e_1)^{n-k}=e_1^k(e_1+e_2)^{n-k}=e_1^k\sum\limits_{i=0}^{n-k}\binom{n-k}{i}e_1^{n-k-i}e_2^i=\sum\limits_{i=0}^{n-k}\binom{n-k}{i}e_1^{n-i}e_2^i.
\end{align*}
Therefore,
\[
S^n(A_2)=
\begin{pmatrix}
1&1&1&\dots&\binom{n}{0}\\
 &1&2&\dots&\binom{n}{1}\\
 & &1&\dots&\binom{n}{2}\\
 & & &\ddots&\vdots\\
& & &&\binom{n}{n}\\
\end{pmatrix},
\]
where empty entries stay for zero. In other words, the columns of
$S^n(A_2)$ are columns of binomial coefficients.
By Lemma~\ref{samegord} we conclude that $S^n(A_2)$ is equivalent to
$A_{n+1}$. This proves the statement of the theorem.
\end{proof}

Let $E$ be a  $2$-dimensional vector bundle over a topological space
$X$. Then there exists an isomorphism
\begin{align*}
S^p(E)\ten S^q(E)\simeq S^{p+q}(E)\oplus(\det E\ten S^{p-1}(E)\ten S^{q-1}(E)).
\end{align*}
This is the Clebsch-Gordan formula. If $\det E$ is the trivial line
bundle, then we have
\(
S^p(E)\ten S^q(E)\simeq S^{p+q}(E)\oplus S^{p-1}(E)\ten S^{q-1}(E)
\), and by iterating one gets
\begin{align}\label{G-Cl}
S^p(E)\ten S^q(E)\simeq S^{p+q}(E)\oplus
S^{p+q-2}(E)\oplus\dots \oplus S^{p-q}(E), \quad p\geqslant q.
\end{align}

\begin{tr}
$F_p\ten F_q\simeq F_{p+q-1}\oplus  F_{p+q-3}\oplus \dots\oplus
  F_{p-q+1}$ for $p\geqslant q$.
\end{tr}
\begin{proof}
Using Theorem~\ref{Th9} and (\ref{G-Cl}) we obtain
\begin{align*}
F_p\ten F_q\simeq &S^{p-1}(F_2)\ten S^{q-1}(F_2)\simeq
S^{p+q-2}(F_2)\oplus S^{p+q-4}(F_2)\oplus\dots \oplus
S^{p-q}(F_2)\simeq\\
&F_{p+q-1}\oplus  F_{p+q-3}\oplus \dots\oplus
  F_{p-q+1}.
\end{align*}
This completes the proof.
\end{proof}
\begin{rem}
The possibility of proving the last theorem using Theorem~\ref{Th9} is exactly what
Atiyah states in remark (1) after Theorem 9 (see~\cite[p. 439]{Atiyah}).
\end{rem}
We have already given (Corollary~\ref{Atiyahbundles}) a description of vector bundles of degree zero
with non-trivial sections. We give now a description of all vector bundles of degree zero.

Consider the function  $\varphi_0(z)=\exp(-\pi i
\tau -2\pi i z)=q^{-1/2}u^{-1}=\varphi(u)$, where $u=e^{2 \pi i
  z}$. It defines the factor of automorphy
\[
e_0(p\tau+q, z)=\exp(-\pi i p^2 \tau- 2\pi i z p)=q^{-\frac{p^2}{2}}u^{-p}
\]
 corresponding to the theta-characteristic $\xi=0$.

\begin{tr}
$\deg E(\varphi_0)=1$, where as above $\varphi_0(z)=\exp(-\pi i
\tau -2\pi i z)=q^{-1/2}u^{-1}=\varphi(u)$.
\end{tr}
\begin{proof}
Follows from Theorem~\ref{degreetheta} for $\xi=0$.
\end{proof}

\begin{tr}
Let $L'\in \cal E(1,d)$. Then there exists $x\in X$ such that
\(
L'\simeq t_x^*E(\varphi_0)\ten E(\varphi_0)^{d-1}.
\)
\end{tr}
\begin{proof}
Since $E(\varphi_0)^d$ has degree $d$, we obtain that there exists $\tilde L\in \cal E(1,0)$ such that
$L'\simeq E(\varphi_0)^d\ten \tilde L$.
We also know that $\tilde
L\simeq t_x^* E(\varphi_0)\ten E(\varphi_0)^{-1}$ (cf. proof of Theorem~\ref{degreetheta} and Theorem~\ref{tr:shift}) for some $x\in
X$. Combining these one obtains
\[
L'\simeq E(\varphi_0)^d\ten t_x^*
E(\varphi_0)\ten E(\varphi_0)^{-1}\simeq  t_x^*E(\varphi_0)\ten
E(\varphi_0)^{d-1}.
\]
This proves the required statement.
\end{proof}

\begin{tr}\label{trivbundleisconst}
The map
\begin{align*}
\bb C^*/<q>\ra \Pic^0(X),\quad \bar a\mto E(a).
\end{align*}
is well-defined and is an isomorphism of groups.
\end{tr}
\begin{proof}
Let $\varphi_0(z)=\exp(-\pi i \tau - 2\pi i z)$ as above. For $x\in X$ consider
$t_x^*E(\varphi_0)$, where the map
\begin{align*}
t_x:X\ra X, \quad y\mto y+x
\end{align*}
is the translation by $x$.
Let $\xi\in \bb C$ be a representative of $x$.  Clearly, $t_x^*E(\varphi_0)$ is defined by
\[{\varphi_0}_\xi(z)=t_{\xi}{\varphi_0}(z)={\varphi_0}(z+{\xi})=\exp(-\pi i \tau -2\pi i z-2\pi i
{\xi})={\varphi_0}(z)\exp(-2\pi i {\xi}).
\](Note that if $\eta$ is another representative of $x$, then
${\varphi_0}_\xi$ and ${\varphi_0}_\eta$ are equivalent.)
Therefore, the bundle $t_x^*E({\varphi_0})\ten
E({\varphi_0})^{-1}$ is defined by
\[
({\varphi_0} _{\xi}
{\varphi_0}^{-1})(z)={\varphi_0}(z)\exp(- 2\pi i{\xi}){\varphi_0}^{-1}(z)=\exp(- 2\pi
i{\xi}).
\]
 Since for any $L\in \cal E(1,0)$ there exists $x\in X$ such that
$L\simeq t_x^* E({\varphi_0})\ten E({\varphi_0})^{-1}$, we obtain $L\simeq
E(a)$ for $a=\exp(- 2\pi i{\xi})\in \bb C^*$, where ${\xi}\in \bb C$ is a
representative of $x$. We proved that any line bundle of degree zero is defined by a constant
function $a\in \bb C^*$.

Vice versa, let $L=E(a)$ for $a\in \bb C^*$. Clearly, there exists
${\xi}\in \bb C$ such that $a=\exp(-2 \pi i {\xi})$. Therefore,
\[
L\simeq E(a)\simeq
L({\varphi_0}_{\xi}{\varphi_0}^{-1})\simeq t_x^* E({\varphi_0})\ten E({\varphi_0})^{-1},
\]
where $x$ is the class of ${\xi}$ in $X$,
which implies that $E(a)$ has degree zero. So we obtained that the line
bundles of degree zero are exactly the line bundles defined by
constant functions.

We have the map
\begin{align*}
\phi: \bb C^*\ra \Pic^0(X), \quad a\mto E(a),
\end{align*}
which is surjective. By Theorem~\ref{tensorPol} it is moreover a
homomorphism of groups. We are looking now for the kernel of this map.

Suppose $E(a)$ is a trivial bundle. Then there exists a holomorphic
function $f:\bb C^*\ra\bb C^*$ such that $f(qu)=af(u)$. Let $f=\sum
f_\nu a^\nu$ be the Laurent series expansion of $f$. Then from
$f(qu)=af(u)$ one obtains
\[
af_\nu=f_\nu q^\nu\text{ for all $\nu \in
  \bb Z$}.
\]
Therefore, $f_\nu(a-q^\nu)=0$ for all $\nu \in \bb Z$.

Since $f\not\equiv 0$, we obtain that there exists $\nu \in \bb Z$ with
$f_\nu\neq 0$. Hence $a=q^\nu$ for some $\nu \in \bb Z$.

Vice versa, if $a=q^\nu$, for $f(u)=u^\nu$ we get
\[
f(qu)=q^\nu
u^\nu=af(u).
\] This means that $E(a)$ is the trivial bundle, which  proves
$\Ker\phi=<q>$. We obtain the required  isomorphism
\begin{align*}
\bb C^*/<q>\ra \Pic^0(X),\quad \bar a\mto E(a).
\end{align*}
This completes the proof.
\end{proof}
\begin{tr}\label{degzerobundle}
For any $F\in \cal E(r,0)$ there exists a unique $\bar a\in \bb C^*/<q>$
such that $F\simeq E(A_r(a))$, where
\[
A_r(a)=
\begin{pmatrix}
a&1&\\
&\ddots&\ddots\\
&&a&1\\
&&&a
\end{pmatrix}
.
\]
\end{tr}
\begin{proof}
By~\cite[Theorem~5]{Atiyah} $F\simeq F_r\ten L$
for a unique $L\in \cal E(1,0)$. Since $F_r\simeq E(A_r)$ and $L\simeq
E(a)$ for  a unique $\bar a\in \bb C^*/<q>$ we get $F\simeq E(A_r\ten
a)$. So $F$ is defined by the matrix
\[
\begin{pmatrix}
a&a&\\
&\ddots&\ddots\\
&&a&a\\
&&&a
\end{pmatrix},
\] where empty entries stay for zeros. It is easy to see that the
Jordan normal form of this matrix is
\[
\begin{pmatrix}
a&1&\\
&\ddots&\ddots\\
&&a&1\\
&&&a
\end{pmatrix}.
\]
This  proves the statement of the theorem.
\end{proof}

\subsection{Vector bundles of arbitrary degree}
Denote by $E_\tau=\bb C/\Gamma_\tau$, where $\Gamma_\tau=\bb Z\tau+\bb
Z$. Consider the $r$-covering
\[
\pi_r:E_{r\tau}\ra E_\tau, \quad [x]\mto [x].
\]
\begin{tr}\label{pullback}
Let $F$ be a vector bundle of rank $n$ on $E_\tau$ defined by
$A(u)=A(1,u)=A(\tau, u)$. Then $\pi_r^*(F)$ is defined by
\[
\tilde A(r\tau, u)=\tilde
A(u)=\tilde A(1,u)\defeq A(r\tau, u)=A(q^{r-1}u)\dots A(qu)A(u).
\]
\end{tr}
\begin{proof}
Consider the following commutative diagram.
\[
\xymatrix
{
&\bb C\ar[ld]_{p_{r\tau}}\ar[rd]^{p_\tau}&\\
E_{r\tau}\ar[rr]^{\pi_r}&&E_\tau
}
\]

Consider the map
\begin{align*}
\bb C\times \bb C^n/\tilde A=E(\tilde A)&\ra
\pi_{r}^*(E(A))=E_{r\tau}\times_{E_\tau}E(A)=\{([z]_{r\tau},\
   [z,v]_\tau)\in E_{r\tau}\times E( A)\}\\
[z,v]_{r\tau}&\mto ([z]_{r\tau},\ [z,v]_\tau).
\end{align*}
It is clearly bijective. It remains to prove that it is
biholomorphic. From the construction of $E(A)$ and
$E(\tilde A)$ it follows that the diagram
\[
\xymatrix
{
\ar[rd] E(\tilde A)&&\\
&{\pi_r^*(E(A))}\ar[r]\ar[d]&E(A)\ar[d]\\
&E_{r\tau}\ar[r]&E_{\tau}
}
\]
locally looks as follows:
\[
\xymatrix
{
U\times\bb C^n\ar[rd]&&\\
&\Delta(U\times U)\times \bb C^n\ar[d]\ar[r]&U\times \bb C^n,\ar[d]\\
&U\ar@{=}[r]&U
}
\xymatrix
{
(z,v)\ar@{|->}[rd]&&\\
&((z,z),v)\ar@{|->}[r]\ar@{|->}[d]&(z,v).\ar@{|->}[d]\\
&z\ar@{|->}[r]&z
}
\]
This proves the required statement.
\end{proof}

\begin{tr}\label{pushforward}
Let $F$ be a vector bundle of rank $n$ on $E_{r\tau}$ defined by
$\tilde A(u)=\tilde A(r\tau, u)$. Then $\pi_{r*}(F)$ is defined by
\(
A(u)=
\begin{pmatrix}
0& I_{(r-1)n}\\
\tilde A(u)&0
\end{pmatrix}
\).

\end{tr}
\begin{proof}
Consider the following commutative diagram.
\[
\xymatrix
{
&\bb C\ar[ld]_{p_{r\tau}}\ar[rd]^{p_\tau}&\\
E_{r\tau}\ar[rr]^{\pi_r}&&E_\tau
}
\]

Let $z\in \bb C$. Consider $y=p_{r\tau}(z)\in E_{r\tau}$ and
$x=p_\tau(z)=\pi_rp_{r\tau}(z)\in E_\tau$.

Choose a point $b\in \bb C$ such that $z\in V_b$, where $V_b$ is the
standard parallelogram at point $b$. Clearly $x\in U_b=p_r(V_b)$ and
we have the isomorphism ${\varphi}_b: U_b\ra V_b$ with
${\varphi}_b(x)=z$.

Consider $\pi_r^{-1}(U_b)=W_b\bigsqcup\dots\bigsqcup W_{b+(r-1)\tau}$, where
$y\in W_b$ and $\pi_r|_{W_{b+i\tau}}:W_{b+i\tau}\ra U_b$ is an
isomorphism for each $0\leqslant i<r$.

We have
\begin{align*}
\pi_{r*}(\cal E(\tilde A))(U_b)=&\cal E(\tilde
A)(\pi^{-1}_r(U_b))=\cal E(\tilde A)( W_b\bigsqcup\dots\bigsqcup
W_{b+(r-1)\tau})=\\&\cal E(\tilde A)(W_b)\oplus\dots\oplus\cal
E(\tilde A)(W_{b+(r-1)\tau}),
\end{align*}
where $\cal E(\tilde A)$ is the sheaf of sections of $E(\tilde A)$.

Choose $a\in \bb C$ such that $z\not\in V_a$, $z\in V_{a+\tau}$. We
have ${\varphi}_a(x)=z+\tau$. As
above, $\pi_r^{-1}(U_a)=W_a\bigsqcup\dots\bigsqcup W_{a+(r-1)\tau}$
and
\begin{align*}
\pi_{r*}(\cal E(\tilde A))(U_a)=&\cal E(\tilde
A)(\pi^{-1}_r(U_a))=\cal E(\tilde A)( W_a\bigsqcup\dots\bigsqcup
W_{a+(r-1)\tau})=\\&\cal E(\tilde A)(W_a)\oplus\dots\oplus\cal
E(\tilde A)(W_{a+(r-1)\tau}).
\end{align*}

Since $g_{ab}(x)=A(\varphi_a(x)-{\varphi}_b(x), \varphi_b(x))$, we obtain
\[
g_{ab}(x)=A(\varphi_a(x)-\varphi_b(x), \varphi_b(x))=A(z+\tau -
z,z)=A(\tau, z).
\] Therefore, to obtain $A(\tau, z)$ it is enough to
compute $g_{ab}(x)$.

Note that $\pi_{r*}(\cal E(\tilde A))_x=\cal E(\tilde
A)_y\oplus\dots\oplus \cal E(\tilde A)_{y+(r-1)\tau}$. Note also that $g_{ab}$ is a map from
\(
\pi_{r*}(\cal E(\tilde A))(U_b)=\cal E(\tilde A)(W_b)\oplus\dots\oplus\cal
E(\tilde A)(W_{b+(r-1)\tau})
\)
 to
\(
\pi_{r*}(\cal E(\tilde A))(U_a)=\cal E(\tilde A)(W_a)\oplus\dots\oplus\cal
E(\tilde A)(W_{a+(r-1)\tau})
\).

One easily sees that $y\in W_b$, $y\in W_{a+(r-1)\tau}$ and
$y+i\tau\in W_{b+i\tau}$, $y+i\tau\in W_{a+(i-1)\tau}$ for
$0<i<r$. Therefore,
\[
g_{ab}(x)=
\begin{pmatrix}
0&\tilde g_{a\ b+\tau}(y+\tau)&\\
\vdots&&\ddots\\
0&&&\tilde g_{a+(r-2)\tau\ b+(r-a)\tau}(y+(r-1)\tau)\\
\tilde g_{a+(r-1)\tau\ b}(y)&0&\hdotsfor{1} &0
\end{pmatrix}.
\]

It remains to compute the entries of this matrix.
Since
\begin{align*}
&\tilde g_{a+(r-1)\tau\ b}(y)=\tilde
A(\tilde\varphi_{a+(r-1)\tau}(y)-\tilde\varphi_b(y),\tilde\varphi_b(y))=\tilde
A(z+r\tau -z,z)=\tilde A(r\tau,z)\text{ and}\\
&\tilde g_{a+(i-1)\tau\ b+i\tau}(y+i\tau)=\tilde
A(\tilde\varphi_{a+(i-1)\tau}(y+i\tau)-\tilde\varphi_{b+i\tau}(y+i\tau),\tilde\varphi_{b+i\tau}(y+i\tau))=\\&\tilde
A(z+i\tau-(z+i\tau)=\tilde A(0,z+i\tau)=I_n,
\end{align*}
one obtains
\[
g_{ab}(x)=
\begin{pmatrix}
0&I_n&\\
\vdots&&\ddots\\
0&&& I_n\\
\tilde A(z)&0&\hdotsfor{1} &0
\end{pmatrix}.
\] Therefore,
\(
A(z)=
\begin{pmatrix}
0&I_n&\\
\vdots&&\ddots\\
0&&& I_n\\
\tilde A(z)&0&\hdotsfor{1} &0
\end{pmatrix}=
\begin{pmatrix}
0& I_{(r-1)n}\\
\tilde A(u)&0
\end{pmatrix}
\). This proves the required statement.
\end{proof}

\begin{lemma}\label{simplelemma}
Let $A_i\in \GL_n(\bb R)$, $i=1,\dots, n$. Then
\[
\prod_{i=1}^r
\begin{pmatrix}
0&I_{(r-1)n}\\
A_i&0
\end{pmatrix}
=\diag(A_r, \dots, A_1)
\]
\end{lemma}
\begin{proof}
Straightforward calculation.
\end{proof}
From Theorem~\ref{pullback} and Theorem~\ref{pushforward} one obtains
the following:
\begin{cor}
Let $E(A)$ be a vector bundle of rank $n$ on $E_{r\tau}$, where $A:\bb
C^*\ra \GL_n(\bb CV)$ is a holomorphic function. Then $\pi^*_r\pi
_{r*}E(A)$ is defined by
\[\diag(A(q^{r-1}u), \dots, A(qu), A(u)).
\] In
other words $\pi^*_r\pi
_{r*}E(A)$ is isomorphic to the direct sum
\[
\bigoplus_{i=0}^{r-1} E(A(q^{i}u)).
\]
\end{cor}
\begin{proof}
We know that $\pi^*_r\pi
_{r*}E(A)$ is defined by $B(r,u)$, where
\[
B(1,u)=
\begin{pmatrix}
0&I_{(r-1)n}\\
A&0
\end{pmatrix}.
\]
Therefore, using Lemma~\ref{simplelemma}, one obtains
\begin{align*}
B(r,u)=&
\begin{pmatrix}
0&I_{(r-1)n}\\
A(q^{r-1}u)&0
\end{pmatrix}
\dots
\begin{pmatrix}
0&I_{(r-1)n}\\
A(qu)&0
\end{pmatrix}
\begin{pmatrix}
0&I_{(r-1)n}\\
A(u)&0
\end{pmatrix}=\\
&\diag(A(q^{r-1}u),\dots, A(qu), A(u)),
\end{align*}
which completes the proof.
\end{proof}
\begin{cor}
Let $L\in \cal E(r,0)$, then $\pi^*_r\pi
_{r*}L=\bigoplus\limits_1^r L$.
\end{cor}
\begin{proof}
Clear, since $L=E(A)$ for a constant matrix $A$ by
Theorem~\ref{degzerobundle}.
\end{proof}

Note that for a covering $\pi_r:E_{r\tau} \ra E_\tau$ the group of deck
transformations $\Deck(E_{r\tau}/E_\tau)$ can be
identified with the kernel $\Ker(\pi_r)$. But $\Ker \pi_r$ is cyclic
and equals $\{1,[q], \dots  [q]^{r-1} \}$, where $[q]$ is a
class of $q=e^{2\pi i\tau}$ in $E_{r\tau}$. Clearly
\[
[q]^*(E(A(u)))=E(A(qu)).
\] Therefore, we get one more corollary.
\begin{cor}
Let $\epsilon$ be a generator of $\Deck(E_{r\tau}/E_\tau)$. Then for a
vector bundle $E$ on $E_{r\tau}$ we have
\[
\pi^*_r\pi_{r*}E=E\oplus \epsilon^* E\oplus\dots\oplus
(\epsilon^{r-1})^* E.
\]
\end{cor}

To proceed we need the following result from~\cite{Oda}(Theorem 1.2, (i)):
\begin{tr*}
Let $\varphi: Y\ra X$ be an isogeny of $g$-dimensional abelian
varieties over a field $k$, and let $L$ be a line bundle on $Y$ such
that the restriction of the map
\[
\Lambda(L): Y\ra \Pic^0(Y),\quad y\mto t_y^*L\ten L^{-1},
\]
to the kernel of
$\varphi$ is an isomorphism. Then $\End(\varphi_* L)=k$ and $\varphi_*L$ is an indecomposable
vector bundle on $X$.
\end{tr*}

\begin{tr}\label{Oda}
Let $L\in\cal E(1,d)$ and let $(r,d)=1$. Then $\pi_{r*}(L)\in \cal E(r,d)$.
\end{tr}
\begin{proof}
It is clear that $\pi_{r*}L$ has rank $r$ and degree $d$. It remains to prove that $\pi_{r*}L$ is
indecomposable.

We have the isogeny $\pi_r:E_{r\tau}\ra E_r$. Since $Y=E_{r\tau}$ is a
complex torus (elliptic curve), $Y\simeq\Pic^0(Y)$ with the identification
$y\leftrightarrow t_y^*E(\varphi_0)\ten E(\varphi_0)^{-1}$. We know that $L=E(\varphi_0)^d\ten \tilde L$ for some $\tilde
L=E(a)\in \cal E(1,0)$, $a\in \bb C^*$. Since $t_y^*(\tilde
L)=t_y^*(E(a))=E(a)=\tilde L$, as in the proof of
Theorem~\ref{trivbundleisconst} one gets
\begin{align*}
\Lambda(L)(y)=&t_y^*(L)\ten L^{-1}=t_y^*(E(\varphi_0)^d\ten \tilde L)\ten
(E(\varphi_0)^d\ten \tilde L)^{-1}=\\
&t_y^*(E(\varphi_0)^d)\ten t_y^*(\tilde L)\ten E(\varphi_0)^{-d}\ten \tilde
L^{-1}=t_y^*(E(\varphi_0)^d)\ten E(\varphi_0)^{-d}=\\
&t_y^*(E(\varphi_0^d)(z))\ten
E(\varphi_0^{-d})=E(\varphi_0^d(z+\eta))\ten
E(\varphi_0^{-d})=\\
&E(\varphi_0^d(z+\eta)\varphi_0^{-d}(z))=E(\exp(-2\pi i d
\eta))=t_{dy}^*(E(\varphi_0))\ten E(\varphi_0)^{-1},
\end{align*}
where $\eta\in \bb C$ is a representative of $y$. This means that the
map $\Lambda(L)$ corresponds to the map
\[
d_Y:E_{r\tau}\ra E_{r\tau}, \quad y\mto dy.
\]
Since  $\Ker\pi_r$ is isomorphic to $\bb Z/r\bb Z$, we
conclude that the restriction of $d_Y$ to $\Ker \pi_r$ is an
isomorphism if and only if $(r,d)=1$. Therefore, using Theorem
mentioned above, we prove the required statement.
\end{proof}
Now we are able to prove the following main theorem:
\begin{tr}\label{main}
(i) Every indecomposable vector bundle $F\in \cal E_{E_\tau}(r,d)$ is of the
form $\pi_{r'*}(L'\ten F_h)$, where $(r,d)=h$, $r=r'h$, $d=d'h$,
$L'\in \cal E_{E_{r'\tau}}(1,d')$.

(ii) Every vector bundle of the form $\pi_{r'*}(L'\ten F_h)$, where
$L'$ and $r'$ are as above, is an
element of $\cal E_{E_\tau}(r,d)$.
\end{tr}
\begin{proof}\footnote{The proof of this theorem uses the ideas from
    lectures presented by Bernd Kreu\ss ler at the University of Kaiserslautern.}
(i) By \cite[Lemma~26]{Atiyah} we obtain $F\simeq E_A(r,d)\ten L$ for some
line bundle $L\in \cal E(1,0)$. By \cite[Lemma~24]{Atiyah} we have
$E_A(r,d)\simeq E_A(r',d')\ten F_h$, hence $F\simeq
E_A(r',d')\ten F_h\ten L$.

Consider any line bundle $\tilde L\in \cal E_{E_{r'\tau}}(1,d')$.
Since by Theorem~\ref{Oda} $\pi_{r'*}(\tilde L)\in \cal E(r',
d')$, it follows from \cite[Lemma~26]{Atiyah}  that there exists a line
bundle $L''$ such that $E_A(r',d')\ten L\simeq \pi_{r'*}(\tilde L)\ten
L''$.

Using the projection formula, we get
\begin{align*}
F\simeq \pi_{r'*}(\tilde L)\ten
L''\ten F_h\simeq \pi_{r'*}(\tilde L\ten \pi^*_{r'}(L'')\ten
\pi^*_{r'}(F_h))\simeq \pi_{r'*}(L' \ten \pi^*_{r'}(F_h))
\end{align*}
 for $L'=\tilde
L\ten \pi^*_{r'}(L'')$.

 Since $F_h$ is defined by a constant matrix we
obtain by Theorem~\ref{pullback} that  $\pi^*_{r'}(F_h)$ is defined
by $f_h^{r'}$, which is has the same Jordan normal form as
$f_h$. Therefore, $\pi^*_{r'}(F_h)\simeq F_h$ and  finally
one gets $F \simeq \pi_{r'*}(L'\ten F_h)$.

(ii) Consider $F=\pi_{r'*}(L'\ten F_h)$. As above
$F_h=\pi_{r'}^*(F_h)$. Using the projection formula  we get
\[
F=\pi_{r'*}(L'\ten F_h) = \pi_{r'*}(L'\ten \pi_{r'}^*(F_h)) =
\pi_{r'*}(L')\ten F_h.
\]
By Theorem~\ref{Oda} $\pi_{r'*}(L')$ is an element from $\cal
E_{E_\tau}(r', d')$. Therefore, $\pi_{r'*}(L')=E_{A}(r',d')\ten L$ for
some line bundle $L\in \cal E_{E_\tau}(1,0)$. Finally we obtain
\[
F=\pi_{r'*}(L')\ten F_h=E_{A}(r',d')\ten L\ten F_h=E_A(r'h,d'h)\ten
L=E_A(r,d)\ten L,
\]
 which means that $F$ is an element of $\cal E_{E_\tau}(r,d)$.
\end{proof}

\begin{rem}
Since any line bundle of degree $d'$ is of the form
$t_x^*E(\varphi_0)\ten E(\varphi_0)^{d'-1}$, Theorem~\ref{main}(i) takes
exactly the form of Proposition 1 from~\cite{Pol}.
\end{rem}

Any line bundle of degree $d'$ over $E_{r\tau}$ is of the form
$E(a)\ten E(\varphi^{d'})$, where $a\in \bb C^*$. Therefore,
$L'\ten F_h=E(a)\ten E(\varphi_0^{d'})\ten
E(A_h)=E(\varphi_0^{d'}A_h(a))$. Using Theorem~\ref{pushforward} we
obtain the following:
\begin{tr}\label{classmatr}
Indecomposable vector bundles of rank $r$ and degree $d$
on $E_{\tau}$ are exactly those defined by the matrices
\[
\begin{pmatrix}
0&I_{(r'-1)h}\\
\varphi_0^{d'}A_h(a)&0
\end{pmatrix},
\]
where $(r,d)=h$, $r'=r/h$, $d'=d/h$,
$\varphi_0(u)=q^{-\frac{r}{2}}u^{-1}$, $q=e^{2\pi i \tau}$, $a\in \bb
C^*$, and
\[
A_h(a)=
\begin{pmatrix}
a&1&\\
&\ddots&\ddots\\
&&a&1\\
&&&a
\end{pmatrix}\in \GL_h(\bb C).
\]
\end{tr}

Note that if $d=0$, we get $h=r$, $r'=1$, and $d'=0$. In this case the
statement of Theorem~\ref{classmatr} is exactly Theorem~\ref{degzerobundle}.


\begin{thebibliography}{10}

\bibitem{Atiyah}
M.~F. Atiyah.
\newblock Vector bundles over an elliptic curve.
\newblock {\em Proc. London Math. Soc. (3)}, 7:414--452, 1957.

\bibitem{Forst}
Otto Forster.
\newblock {\em {Riemannsche Fl\"achen.}}
\newblock {Heidelberger Taschenb\"ucher. Band 184. Berlin-Heidelberg-New York:
  Springer-Verlag. X, 223 S. mit 6 Fig.}, 1977.

\bibitem{GrHar}
Phillip Griffiths and Joseph Harris.
\newblock {\em {Principles of algebraic geometry.}}
\newblock {Pure and Applied Mathematics. A Wiley-Interscience Publication. New
  York etc.: John Wiley \& Sons. XII, 813 p.}, 1978.

\bibitem{MyGermanDiplom}
Oleksandr Iena.
\newblock Vector bundles on elliptic curves and factors of automorphy.
\newblock Diplomarbeit, Technische Universut\"at Kaiserslautern, Germany,
  August 2005.

\bibitem{Lange}
Herbert Lange and Christina Birkenhake.
\newblock {\em {Complex abelian varieties.}}
\newblock {Grundlehren der Mathematischen Wissenschaften. 302. Berlin:
  Springer- Verlag. viii, 435 p.}, 1992.

\bibitem{MumThetaI}
David Mumford.
\newblock {\em {Tata lectures on theta. I: Introduction and motivation: Theta
  functions in one variable. Basic results on theta functions in several
  variables. With the assistance of C. Musili, M. Nori, E. Previato, and M.
  Stillman.}}
\newblock {Progress in Mathematics, Vol. 28. Boston - Basel - Stuttgart:
  Birkh\"auser. XIII, 235 p.}, 1983.

\bibitem{MumThetaII}
David Mumford.
\newblock {\em {Tata lectures on theta. II: Jacobian theta functions and
  differential equations. With the collaboration of C. Musili, M. Nori, E.
  Previato, M. Stillman, and H. Umemura.}}
\newblock {Progress in Mathematics, Vol. 43. Boston-Basel-Stuttgart:
  Birkh\"auser. XIV, 272 p.}, 1984.

\bibitem{MumThetaIII}
David Mumford, Madhav Nori, and Peter Norman.
\newblock {\em {Tata lectures on theta. III.}}
\newblock {Progress in Mathematics. 97. Boston, MA etc.: Birkh\"auser. vii, 202
  p.}, 1991.

\bibitem{Oda}
T.~Oda.
\newblock {Vector bundles on an elliptic curve.}
\newblock {\em Nagoya Math. J.}, 43:41--72, 1971.

\bibitem{Pol}
Alexander Polishchuk and Eric Zaslow.
\newblock {Categorical mirror symmetry: The elliptic curve.}
\newblock {\em Adv. Theor. Math. Phys.}, 2(2):443--470, 1998.

\bibitem{Serre}
Jean-Pierre Serre.
\newblock {G\'eom\'etrie alg\'ebrique et g\'eom\'etrie analytique.}
\newblock {\em Ann. Inst. Fourier}, 6:1--42, 1956.

\end{thebibliography}
\def\cprime{$'$} \def\cprime{$'$}

\end{document}